\documentclass{amsart}

\usepackage{amsmath}
\usepackage{amsthm}
\usepackage{hyperref}
\usepackage{amsfonts,graphics,amsthm,amsfonts,amscd,latexsym}
\usepackage{epsfig}
\usepackage{flafter}
\usepackage{mathtools}
\usepackage{comment}
\usepackage{stmaryrd}
\usepackage[normalem]{ulem}

\usepackage{mathabx,epsfig}

\hypersetup{
    colorlinks=true,    
    linkcolor=blue,          
    citecolor=blue,      
    filecolor=blue,      
    urlcolor=blue           
}
\usepackage{tikz}
\usetikzlibrary{graphs,positioning,arrows,shapes.misc,decorations.pathmorphing}

\tikzset{
    >=stealth,
    every picture/.style={thick},
    graphs/every graph/.style={empty nodes},
}

\tikzstyle{vertex}=[
    draw,
    circle,
    fill=black,
    inner sep=1pt,
    minimum width=5pt,
]
\usepackage[position=top]{subfig}
\usepackage{amssymb}
\usepackage{color}

\calclayout

\usetikzlibrary{decorations.pathmorphing}
\tikzstyle{printersafe}=[decoration={snake,amplitude=0pt}]

\newcommand{\Cl}{\operatorname{Cl}}
\newcommand{\Pic}{\operatorname{Pic}}

\newcommand{\rank}{\operatorname{rank}}

\newcommand{\supp}{\operatorname{supp}}

\newcommand{\schbase}{\mathcal{S}}

\newcommand{\pp}{\mathbb{P}}

\newcommand{\qq}{\mathbb{Q}}
\newcommand{\zz}{\mathbb{Z}}

\newcommand{\rr}{\mathbb{R}}

\newcommand{\Exc}{\mathrm{Exc}}

\def\O#1.{\mathcal {O}_{#1}}			
\def\pr #1.{\mathbb P^{#1}}				
\def\af #1.{\mathbb A^{#1}}			
\def\ses#1.#2.#3.{0\to #1\to #2\to #3 \to 0}	
\def\xrar#1.{\xrightarrow{#1}}			
\def\K#1.{K_{#1}}						
\def\bA#1.{\mathbf{A}_{#1}}			
\def\bM#1.{\mathbf{M}_{#1}}				
\def\bL#1.{\mathbf{L}_{#1}}				
\def\bB#1.{\mathbf{B}_{#1}}				
\def\bK#1.{\mathbf{K}_{#1}}			
\def\subs#1.{_{#1}}					
\def\sups#1.{^{#1}}

\usepackage{tikz}
\usetikzlibrary{matrix,arrows,decorations.pathmorphing}

  \newtheorem{theorem}{Theorem}[section]
  \newtheorem{lemma}[theorem]{Lemma}
  \newtheorem{proposition}[theorem]{Proposition}
  \newtheorem{corollary}[theorem]{Corollary}
  
  \newtheorem{conjecture}[theorem]{Conjecture}

  \newtheorem{definition}[theorem]{Definition}
  \newtheorem{example}[theorem]{Example}
  
  \newtheorem{problem}[theorem]{Problem}
  \newtheorem{question}[theorem]{Question}

\theoremstyle{remark}

\numberwithin{equation}{section}

\usepackage[all]{xy}

\begin{document}

\title[Birational complexity and dual complexes]{Birational complexity and dual complexes}

\author[M.~Mauri]{Mirko Mauri}
\address[Mirko Mauri]{\'{E}cole Polytechnique, Rte de Saclay, 91120 Palaiseau, France}
\email{mirko.mauri@polytechnique.edu}

\author[J.~Moraga]{Joaqu\'in Moraga}
\address[Joaqu\'in Moraga]{UCLA Mathematics Department, Box 951555, Los Angeles, CA 90095-1555, USA
}
\email{jmoraga@math.ucla.edu}

\thanks{MM was supported by the Hausdorff Institute of Mathematics in Bonn, the Institute of Science and Technology Austria, and \'{E}cole Polytechnique in France. This project has received funding from the European Union’s Horizon 2020 research and innovation programme
under the Marie Skłodowska-Curie grant agreement No 101034413.
}

\maketitle

\begin{abstract}
We introduce the notion of {\em birational complexity}
of a log Calabi--Yau pair. This invariant measures how far the log Calabi--Yau pair is from being birational to a toric pair. We study fundamental properties of the new invariant, with a particular focus on the geometry of dual complexes.
\end{abstract}

\setcounter{tocdepth}{1} 
\tableofcontents

\section{Introduction} 
Calabi--Yau and Fano varieties are fundamental building blocks of algebraic varieties, and log Calabi--Yau pairs are logarithmic or degenerate version of both of them. 

\begin{definition} \emph{A \emph{log Calabi--Yau pair} $(X,B)$ consists of a proper variety $X$ and an effective $\rr$-divisor $B$ such that the pair $(X,B)$ is log canonical and $K_X+ B$ is $\rr$-linearly trivial.} 
\end{definition}
If $B=0$, we recover the classical definition of Calabi--Yau variety. At the opposite extreme, a Fano variety $X$ and an ample anti-canonical divisor $B \sim_{\rr} -K_{X}$ form a log Calabi--Yau pair $(X,B)$, provided that the pair is log canonical. 
Actually, these two extreme cases, i.e., $B=0$ or $B$ fully supporting an ample divisor, are the fundamental building blocks of any log Calabi--Yau pair. We prove this general structural result in Theorem~\ref{theorem:supp-amp}, thus refining previous work of Koll\'{a}r--Xu and Filipazzi--Svaldi; see also Sections~\ref{sec:ample} and \ref{sec:positive}.

In a degeneration process, the geometry of Calabi--Yau varieties converts into the combinatorics of log Calabi--Yau pairs, so it is crucial to investigate this aspect of log Calabi--Yau pairs. The combinatorial information of how the irreducible components of a log canonical pair $(X,B)$ intersect is encoded in a CW-complex called dual complex $\mathcal{D}(X,B)$; see Definition \ref{defn:dualcomplex}. 
Dual complexes naturally appear in the classification of singularities and in mirror symmetry; see \cite{Kol13, dFKX17, KX16, KS2006, Li2023} for a non-exhaustive list. 

For instance, according to mirror symmetry, Calabi--Yau varieties should come in pairs $(Y, \hat{Y})$, and one should reconstruct $\hat{Y}$ by dualising a special Lagrangian torus fibration $f\colon Y \to S$, called SYZ fibration. Kontevich and Soibelman conjectured that $S$ should be a manifold (or at least an orbifold). Equivalently, $S$ should be locally modelled on a ball, alias a cone over a sphere (or a finite quotient of a sphere). This sphere is expected to have a combinatorial interpretation: it should be the dual complex of a log Calabi--Yau pair, appearing in a degeneration of $Y$. 

Later, Koll\'{a}r and Xu asked whether this is actually a general phenomenon of all log Calabi--Yau pairs; see \cite[Question 4]{KX16}.

\begin{conjecture} [Algebro-geometric version of the Poincar\'{e} conjecture] \label{conj:poinc}
The dual complex of a log Calabi--Yau pair is homeomorphic to a finite quotient of a sphere.  
\end{conjecture}
More informally, a combinatorial Calabi--Yau variety should be a finite quotient of a sphere.
Conjecture~\ref{conj:poinc} should be regarded as an algebro-geometric version of the Poincar\'{e} conjecture. Indeed, the log Calabi--Yau condition $K_X+B \sim 0$ is equivalent to the cohomological assumption $h^0(X, K_X+B)=h^{0}(X, -(K_X+B))=1$, and Conjecture~\ref{conj:poinc} requires that this cohomological datum should prescribe a spherical homeomorphism type for the dual complex, in the same way as for the classical Poincar\'{e} conjecture. See Section~\ref{sec:bircompldualcomplex} for a brief account on the state of the art of Conjecture~\ref{conj:poinc}. See also \cite{Simpson2016, MMS2022} for the role of Conjecture~\ref{conj:poinc} in the geometric P=W conjecture. 

In this paper we provide new evidence to support Conjecture~\ref{conj:poinc}. 

\begin{theorem}\label{thm:dualcomplex} \emph{(weaker version of Theorem~\ref{theorem:dc-smooth})}
If $(X,B)$ is 
a $\qq$-factorial log Calabi--Yau pair, the sum of the coefficients of $B$ is greater than the Picard number of $X$, i.e., $|B| > \rho(X)$, and $\mathcal{D}(X, B)$ is a PL-manifold, then $\mathcal{D}(X, B)$ is homeomorphic to a sphere or a disk.
\end{theorem}

If $K_{X}+B \sim 0$, the log Calabi--Yau condition grants that $\mathcal{D}(X, B)$ is a rational homology sphere; see \cite[\S 4]{KX16}. Even if $\mathcal{D}(X, B)$ is a simply-connected PL-manifold, it is not a priori clear that it is a sphere. As a counterexample,  consider for instance the 5-dimensional Wu manifold $SU(3)/SO(3)$. By the classical Poincar\'{e} conjecture, the only obstruction is the torsion of the integral homology, which vanishes under our algebro-geometric assumption $|B| > \rho(X)$. This numerical condition is an upper bound on the so-called complexity of the pair $(X, B)$, as discussed in the following section.

\subsection{Birational complexity}
A prototypical example of log Calabi--Yau pair is a toric variety with its toric boundary. Despite how restrictive the definition is, quite surprisingly, toric varieties are ubiquitous in algebraic geometry. In particular,  the study of toric pairs often guides and enlightens the study of general log Calabi--Yau pairs, so it is important to measure how close a log Calabi--Yau pair is from being a toric pair. The complexity is an invariant that serves this purpose. Specifically, the complexity of the log Calabi--Yau pair $(X,B=\sum^r_{i=1} a_i B_i)$ is the non-negative real number 
\[
c(X,B) \coloneqq  \dim(X) + \rank \Cl(X)_\qq -\sum^r_{i=1} a_i,
\]
and the pair $(X,\lfloor B\rfloor)$ is toric\footnote{A toric pair $(X,\Delta)$ consists of a toric variety $X$ and a divisor whose prime components are torus invariant. The sum of all reduced torus invariant divisors is called toric boundary.} if $c(X,B)$ vanishes; see~\cite{BMSZ18}.

In~\cite{MS21}, Svaldi and the second author gave a characterization of toric singularities
via a local variant of the complexity.
Recently, Gongyo and the second author 
proved statements related to the Kobayashi--Ochiai theorem for surfaces using a variant of the complexity for generalized pairs~\cite{GM23}.
In order to achieve the aforementioned results,
many variants of the complexity have been introduced:
fine complexity, orbifold complexity, and generalized complexity. 

In this article, we introduce the notion of {\em birational complexity} of a log Calabi--Yau pair $(X,B)$, denoted $c_{\rm bir}(X,B)$.
It is the minimum among all the complexities of pairs crepant birational to $(X,B)$.
In particular, we study how this invariant reflects in the geometry of $X$ and the dual complex $\mathcal{D}(X,B)$. 

The relation between birational complexity and dual complexes stems from the following simple observation. Out of the combinatorial data defining a toric variety, one can easily read that the dual complex of the toric boundary is a sphere: the cone over the dual complex gives the fan of the toric variety. It is then natural to expect that the closer a log Calabi--Yau pair is to be crepant birational to a toric pair (namely the lower its birational complexity is) the closer its dual complex is to be a sphere. This is indeed the intuition behind Theorem~\ref{thm:dualcomplex}.

\subsection{Birational complexity and coregularity} 
We first investigate
upper bounds for the birational complexity
in terms of the coregularity. Let $(X,B)$ be a log Calabi--Yau pair. In Lemma~\ref{lem:computing-bcomp}, we show that there exists a model that computes the birational complexity
and it can be chosen to be $\qq$-factorial and dlt. The following theorem shows that the birational complexity
is bounded above by the sum of the coregularity and the dimension.

\begin{theorem}\label{theorem:bcomp-bound}
Let $X$ be a variety of Fano type 
of dimension $n$ 
and $(X,B)$ be a log Calabi--Yau pair.
Then, the following inequalities hold:
\begin{equation}\label{eq:1}
0\leq c_{\rm bir}(X,B)\leq {\rm coreg}(X,B)+n
\end{equation} 
and 
\begin{equation}\label{eq:2}
0\leq c_{\rm bir}(X,B) < 2n.
\end{equation} 
\end{theorem} 

In Example~\ref{ex:arbirary-bcomp},
we show that for any positive integer $n$
and $c\in \{0,\dots,n-1\}$,
there exists a log Calabi--Yau pair $(X,B)$
of dimension $n$, 
coregularity $c$, and birational complexity at least $n$.
This shows that inequality~\eqref{eq:1} is sharp at least in coregularity zero.
On the other hand,
in Example~\ref{ex:maximal-bcomp}, 
we show that for every positive integer $n$ and $\epsilon >0$, there exists a log Calabi--Yau pair $(X_\epsilon,B_\epsilon)$ of dimension $n$ with 
$2n-c_{\rm bir}(X_\epsilon,B_\epsilon)<\epsilon$.
Example~\ref{ex:smooth-cy}, shows that the statement of the theorem does not hold if $X$ is not a variety of Fano type.

In~\cite{FMM22}, Filipazzi and the authors proved that a log Calabi--Yau pair 
with standard coefficients and coregularity zero has index at most $2$.
In~\cite{FFMP22}, Filipazzi, Figueroa, Peng and the second author proved that a Fano variety of absolute coregulariy zero admits either a $1$-complement or a $2$-complement.
As a corollary, we obtain the following result.

\begin{corollary}\label{introcor:bcomp-coreg-0}
Let $X$ be a variety of Fano type variety and absolute coregularity zero.
Up to a $2$-to-$1$ cover of $X$ there exists a $1$-complement $B$
with $(X,B)$ of coregularity zero for which:
\[
c_{\rm bir}(X,B)\in \{0,\dots,n\}.
\]
\end{corollary}

By Lemma~\ref{lem:computing-bcomp} and ~\cite[Theorem 1.2]{BMSZ18},
we get the following characterization of pairs 
with birational complexity zero:

\begin{theorem}\label{theorem:bcomp-zero}
Let $(X,B)$ be a log Calabi--Yau pair with $K_X+B\sim 0$. 
Then
$c_{\rm bir}(X,B)=0$
if and only if 
there exists a crepant
birational map
\[
j\colon (\pp^n,H_1+\dots+H_{n+1})
\dashrightarrow (X,B),
\]
where the $H_i$'s are the hyperplane coordinates of $\mathbb{P}^n$.
In particular, 
${\rm coreg}(X,B)=0$
and $\mathcal{D}(X,B)\simeq_{\rm PL} S^{n-1}$.
\end{theorem}

If $X$ is a Fano surface of absolute coregularity zero, up to a $2$-to-$1$ cover of $X$, there exists a $1$-complement $B$ such that $c_{\rm bir}(X,B)=0$.
In particular, by Theorem~\ref{theorem:bcomp-zero}, the log pair $(X,B)$ is crepant equivalent to $(\mathbb{P}^2,H_1+H_2+H_3)$ and $\mathcal{D}(X,B)\simeq S^1$; see also \cite[Proposition 1.3]{GHK2015}. In general, in dimension $2$, a log pair of coregularity zero has
birational complexity zero up to a $2$-to-$1$ cover.
This is not true anymore in dimension $3$; see for instance \cite{K2020}.
In the context of Corollary~\ref{introcor:bcomp-coreg-0}, it is not clear whether we can attain all the possible values between $0$ and $n$.

\subsection{Ample divisors in the boundary}\label{sec:ample}
We now explain the relation 
between the birational complexity
of a log Calabi--Yau pair
and the existence
of positive divisors in the boundary divisor.
In~\cite{KX16}, Koll\'ar and Xu proved that a log Calabi--Yau pair $(X,B)$ of coregularity zero
is birational to a log Calabi--Yau pair $(X_0,B_0)$ where $X_0$ is a variety of Fano type. 
To show this, they prove that on a crepant birational model $(X_0,B_0)$ of $(X,B)$ the divisor $\lfloor B_0\rfloor$ supports a big divisor.
In this direction, we show the existence of a $\qq$-factorial dlt crepant birational model $(X_0,B_0)$ of $(X,B)$ on which $\lfloor B_0\rfloor$ supports an ample divisor.
In the case of higher coregularity, 
a similar relative statement holds.

\begin{theorem}\label{theorem:supp-amp}
Let $(X,B)$ be a log Calabi--Yau pair.
Then, there exist:
\begin{enumerate}
\item[(i)] a $\qq$-factorial dlt crepant birational model $(X',B')$ of $(X,B)$, and 
\item[(ii)] a fibration $X'\rightarrow Z$,
\end{enumerate} 
such that the following conditions hold:
\begin{enumerate}
\item every log canonical center of $(X',B')$ dominates $Z$, and
\item the divisor $\lfloor B'\rfloor$ fully supports an effective divisor that is ample over $Z$.
\end{enumerate}
\end{theorem} 

In the case that the log Calabi--Yau pair $(X,B)$ has coregularity zero, we can take $Z$ in the previous statement to be a point.
In~\cite[Theorem 4.9]{KX16}, Koll\'ar and Xu proved a version of the previous theorem in which $\lfloor B'\rfloor$ 
fully supports a semiample and big divisor $H$ over $Z$. 
The so-called {\em Koll\'ar--Xu model} of the pair $(X,B)$.
In their setting, we can take the ample model of $H$ and lose the $\mathbb{Q}$-factoriality of $X$. 
This statement was later proved in the context
of generalized pairs by Filipazzi and Svaldi~\cite{FS20}.
It is not clear whether the techniques of~\cite{KX16,FS20} can be modified to get both a $\qq$-factorial model and $\lfloor B\rfloor$ supporting an ample divisor.
Our approach uses the work of Koll\'ar--Xu and Filipazzi--Svaldi as a starting point, 
and we argue by running the MMP and the $2$-ray game. 

Our next statement is an improvement of Theorem~\ref{theorem:supp-amp}.
Whenever the birational complexity is not maximal,
we can arrange the ample divisor to be properly supported on $\lfloor B'\rfloor$.
This apparently innocuous improvement will be pivotal in our results about dual complexes.

\begin{theorem}\label{theorem:proper-supp-amp}
Let $(X,B)$ be a log Calabi--Yau pair of dimension $n$.
Assume that $K_X+B\sim 0$ and
$c_{\rm bir}(X,B)<n$.
Then, there exist:
\begin{enumerate}
\item[(i)] a $\qq$-factorial dlt crepant birational model $(X',B')$ of $(X,B)$, and 
\item[(ii)] a fibration $X'\rightarrow Z$,
\end{enumerate} 
such that the following conditions hold:
\begin{enumerate}
\item every log canonical center of $(X',B')$ dominates $Z$, 
\item the divisor $\lfloor B'\rfloor$ properly supports an effective 
divisor which is ample over $Z$, and
\item the dual complex $\mathcal{D}(X',B')$ is simplicial.
\end{enumerate} 
\end{theorem}

In Example~\ref{ex:supp-ample}, we show that the previous statement does not hold if 
we drop
$K_X+B\sim 0$.
Indeed, in this case the divisor $\lfloor B'\rfloor$ may be empty or irreducible in every possible crepant model $(X',B')$ of $(X,B)$.

\subsection{Positive divisors in the boundary}\label{sec:positive} In a similar vein than the previous section, 
we can study 
the existence of {\em positive} divisors
in the boundary of a log Calabi--Yau pair $(X,B)$.
The answer depends on the notion of {\em positivity} that we consider.
If $B$ is reduced, then the number of components of $B$ is bounded above by $\dim X+\rho(X)$ and the log pair is toric if the equality holds.
In the previous statement no positivity assumption is required for the components of $B$.
If instead we study the number of positive components of $B$, then the situation becomes more restrictive. 
In what follows, we introduce some results about the number of {\em positive} divisors that $B$ can support. Further, we show some implications on the dual complex $\mathcal{D}(X,B)$ when the maximal number of positive divisors on $B$ is achieved.

\begin{definition}
A {\em decomposition} of an effective divisor $B$ is a formal sum $\sum_{i\in I} B_i$ of effective divisors such that $\bigcup_{i\in I}{\rm supp}(B_i)\subseteq {\rm supp}(B)$ and $B_i\wedge B_j=0$ for $i,j\in I$ and $i\neq j$. We say that a decomposition is {\em into big divisors} if each $B_i$ is a big divisor.
Analogously, we say that a decomposition is {\em into movable divisors} if each $B_i$ is movable.
\end{definition}
Our first result in this direction is a bound on the number of components of decompositions into big divisors. This is a generalization of \cite[Theorem 4.1]{Mauri2020}.

\begin{theorem}\label{theorem:big-decomp}
Let $X$ be an $n$-dimensional klt variety and let $(X,B)$ be a log Calabi--Yau pair.
Let $\sum_{i\in I} B_i$ be a decomposition of $\lfloor B \rfloor$ into big divisors.
Then, we have that $|I|\leq n+1$.
Furthermore, if the equality holds, then there exists a crepant birational contraction
$(X,B)\dashrightarrow (T,B_T)$ where $T$ is a weighted projective space with its toric boundary $B_T$. 
In particular, $\mathcal{D}(X,B)$ is PL-homeomorphic to the boundary of an $n$-dimensional standard simplex. 
\end{theorem}

In Example~\ref{ex:wpi-blow-up}, we provide examples of $n$-dimensional log Calabi--Yau pairs $(X,B)$ such that $B$ 
admits a decomposition into $n+1$ big divisors
and the Picard rank $\rho(X)$ is arbitarily large.
This means that it is not possible to control the number of blow-downs in the birational contraction $X\dashrightarrow T$.
In a similar direction, 
we give a bound on the number of components of a movable decomposition of the boundary divisor.

\begin{theorem}\label{theorem:movable-decomp}
Let $X$ be an $n$-dimensional klt variety and let $(X,B)$ be a log Calabi--Yau pair.
Let $\sum_{i\in I} B_i$ be a decomposition of $\lfloor B\rfloor$ into movable divisors.
Then, we have that $|I|\leq 2n$.
Furthermore, if the equality holds, then we have an isomorphism 
$(X,B)\simeq ((\mathbb{P}^1)^n,B_T)/A$, 
where $B_T$ is the toric boundary
and $A\leqslant \mathbb{G}_m^n \leqslant {\rm Aut}((\mathbb{P}^1)^n,B_T)$. 
In particular, $\mathcal{D}(X,B)$ is isomorphic to an $n$-dimensional orthoplex.
\end{theorem}

Alternatively, the toric variety $((\mathbb{P}^1)^n,B_T)/A$ in Theorem~\ref{theorem:movable-decomp} can be characterized by the fact that its moment polytope is a parallelotope.

\subsection{Dimension of the dual complex} 

In Theorem~\ref{theorem:bcomp-vs-dim}, we explore connections
between 
birational complexity, 
decompositions into movable divisors, and
the dimension of the dual complex
of a log Calabi--Yau pair. 

\begin{theorem}\label{theorem:bcomp-vs-dim}
Let $(X,B)$ be a log Calabi--Yau pair of dimension $n$ with $K_X+B\sim 0$.
Let $c=c_{\rm bir}(X,B)$.
Then, there exist a $\mathbb{Q}$-factorial
crepant birational model $(X',B')$ of $(X,B)$ 
and a fibration $X'\rightarrow Z$
satisfying the following conditions:
\begin{itemize}
\item every log canonical center of $(X',B')$ dominates $Z$,
\item the divisor $\lfloor B'\rfloor$ admits a decomposition into $n-c$ 
effective divisors that are movable over $Z$, and 
\item we may assume that the intersection of all the movable divisors in the decomposition is non-trivial.
\end{itemize}
In particular, we have that $\dim \mathcal{D}(X,B)\geq n-c-1$.
\end{theorem}

\subsection{Birational complexity and dual complexes}\label{sec:bircompldualcomplex}
The main goal of this paper is to describe the relation between the dual complex of a log Calabi--Yau pair and its birational complexity.

The following folklore conjecture is stated as a question in \cite[Question 4]{KX16}, and it could be considered an algebro-geometric version of the Poincaré conjecture.
\begin{conjecture}[Algebro-geometric version of the Poincaré conjecture]\label{introconj:dc}
Let $(X,B)$ be an $n$-dimensional log Calabi--Yau pair. Then $\mathcal{D}(X,B)\simeq_{\rm PL} S^k/G$ where $k\leq n-1$ and $G$ is a finite subgroup of the orthogonal group $O(k)$.
\end{conjecture}
Conjecture~\ref{introconj:dc} is known up to dimension $4$ due to the work of Koll\'ar and Xu~\cite{KX16}, and in dimension $5$ if further $(X, B)$ is an snc pair. In higher dimension, there exists only partial evidence:
\begin{enumerate}
\item the conjecture holds if the dlt log Calabi--Yau pair $(X, B)$ is endowed with the structure of Mori fiber space $(X, B) \to Z$ with $\min\{\dim Z, \rho(Z)\} \leq 2$; see \cite{Mauri2020}.
    \item $\mathcal{D}(X,B)$ has the rational homology of the sphere or of the point; see \cite[Theorem 2.(2)]{KX16}. On the other hand, the torsion of the integral homology remains quite mysterious.
\item The fundamental group of $\mathcal{D}(X, B)$ is finite, and the fundamental groups of $\mathcal{D}(X,B)$ in any fixed dimension form a finite set; see \cite[Theorem 2.(3)]{KX16}, \cite[Theorem 2]{Bra20}, and \cite[Corollary 1]{Mor21}.
\item If $B$ is reduced, then $\mathcal{D}(X, B)$ is orientable if and only if $K_{X}+B \sim 0$; see \cite[Corollary 4]{FMM22}.
\end{enumerate}

Our main theorem here states that 
the dual complex of a log Calabi--Yau pair is union of two collapsible subcomplexes if the birational complexity is strictly less than the dimension of the underlying variety.

\begin{theorem}\label{theorem:dc-nonmax-bc}
Let $(X,B)$ be a log Calabi--Yau pair of dimension $n$.
If $c_{\rm bir}(X,B)<n$, then $\mathcal{D}(X,B)$ is union of two collapsible subcomplexes.
\end{theorem}

In Example~\ref{ex:max-bcomp-dual},
we show that the previous statement does not necessarily hold if $c_{\rm bir}(X,B)=n$. In general, this is not sufficient to conclude that $\mathcal{D}(X,B)$ is a sphere, unless $\mathcal{D}(X,B)$ is a PL manifold, i.e., the link of any simplex is a PL sphere; see Example~\ref{ex:suspension} and
Proposition~\ref{prop:PL-sphere}.
Henceforth, we obtain the following result for 
Calabi--Yau pairs with smooth dual complexes.

\begin{theorem}\label{theorem:dc-smooth}
Let $(X,B)$ be a log Calabi--Yau pair of dimension $n$
with $c_{\rm bir}(X,B)<n$. If $\mathcal{D}(X,B)$ is a PL-manifold, then either 
$\mathcal{D}(X,B)\simeq S^k$ or
$\mathcal{D}(X,B)\simeq_{\rm PL} D^k$. with $k\leq n-1$. 
In the former case, $\mathcal{D}(X,B)\simeq_{\mathrm{PL}} S^k$ if further $\dim \mathcal{D}(X,B)\neq 4$.
\end{theorem} 

Theorem~\ref{theorem:dc-smooth} provides new evidence for Conjecture~\ref{introconj:dc} in higher dimension.
In particular, this theorem states that log Calabi--Yau pairs with interesting dual complexes (neither a sphere nor a disk)
must have maximal birational complexity.
We expect that log Calabi--Yau pairs with maximal birational complexity are rather special.
Theorem~\ref{theorem:dc-smooth} motivates the two following problems:

\begin{problem}\label{probl:resol-sing}
Let $(X,B)$ be a log Calabi--Yau pair. Prove that the existence of a finite Galois (possibly ramified) cover $(X',B')\rightarrow (X,B)$ for which $\mathcal{D}(X',B')$ is a PL manifold. 
\end{problem}

\begin{problem}\label{probl:dc-max-bc}
Let $(X,B)$ be a log Calabi--Yau pair of
coregularity $0$ and maximal birational complexity, i.e., $c_{\rm bir}(X,B)=\dim X$.
Assume that $\mathcal{D}(X,B)$ is a PL manifold. Prove that $\mathcal{D}(X,B)$ is a finite quotient of a sphere.
\end{problem} 

Theorem~\ref{theorem:dc-smooth}, together with a positive answer of Problem~\ref{probl:resol-sing} and Problem~\ref{probl:dc-max-bc}, would imply Conjecture~\ref{introconj:dc}.
While Problem~\ref{probl:resol-sing} seems to be out of reach for the time being, we expect that similar techniques than the ones in this paper may lead to a positive solution of Problem~\ref{probl:dc-max-bc}.

As an application of Theorem~\ref{theorem:dc-nonmax-bc}, we can prove that in each dimension $n\geq 3$
there are log Calabi--Yau pairs
of coregularity zero and maximal birational complexity.

\begin{theorem}\label{theorem:max-bir-comp}
Let $n\geq 3$ be a positive integer.
There exists an $n$-dimensional log Calabi--Yau pair
$(X,B)$ satisfying that the following conditions:
\begin{enumerate}
\item $i(K_X+B)\sim 0$,
\item ${\rm coreg}(X,B)=0$, and 
\item $c_{\rm bir}(X,B)=n$,
\end{enumerate}
where $i=1$ if $n$ odd 
and $i=2$ if $n$ is even.
\end{theorem}
It is not clear what birational complexities can be achieved among $n$-dimensional log Calabi--Yau pairs of coregularity zero. The previous theorem implies that $0$ and $n$ are always achieved in dimension at least $3$.

\subsection{Effective cone of maximal birational complexity}

We conclude the introduction 
by stating some results about log Calabi--Yau pairs of coregularity 
zero with maximal birational complexity.
Up to a birational transformation, 
we give a structural theorem about the cone of effective divisors.

\begin{theorem}\label{theorem:eff-cone-max-bir}
Let $(X,B)$ be a log Calabi--Yau pair of coregularity zero and dimension $n$.
Assume that $c_{\rm bir}(X,B)=n$.
There exists a crepant birational model $(X',B')$ of $(X,B)$ satisfying the following conditions:
\begin{enumerate}
\item the variety $X'$ is $\qq$-factorial of Picard rank $\rho$,
\item the pair $(X',B')$ is dlt,
\item the cone of divisors is simplicial, 
\item the prime components of $B'$ generate the effective cone of divisors.
\end{enumerate}
\end{theorem}

To summarize, if $(X,B)$ is a log Calabi--Yau pair
of coregularity zero and dimension $n$, then there exists a crepant birational pair $(X', B')$ such that at least one of the following statements holds: \vspace{0.2 cm}
\begin{center}
\setlength{\tabcolsep}{10pt}
\renewcommand{\arraystretch}{1.5}
\begin{tabular}{ |c|c| } 
\hline 
$c_{\rm bir}(X,B)=0$ & $(X', \lfloor B'\rfloor )$ is a toric pair\\ 
\hline
 $c_{\rm bir}(X,B)< n$ &  $\mathcal{D}(X',B')$ is a union of two collapsible subcomplexes\\ 
 \hline 
 $c_{\rm bir}(X,B)=n$ & $\overline{\mathrm{Eff}}(X')$ is generated by the prime components of $B'$\\ 
 \hline
\end{tabular}
\end{center}
\vspace{0.2 cm}

\subsection{Acknowledgement} We warmly thank
Louis Esser,
Stefano Filipazzi, 
Maggie Miller, 
Sucharit Sarkar, and 
Roberto Svaldi 
 for useful conversations. 
\section{Preliminaries}

In this section, we introduce some preliminary definitions and results regarding generalized pairs, Fano and Calabi--Yau pairs,
dual complexes, complexity, and PL spheres.
We work over an algebraically closed field $\mathbb{K}$ of characteristic zero.

\subsection{Generalized singularities} In this section, we recall the concept of generalized pairs, generalized singularities, and some lemmata.

\begin{definition}
{\em 
A {\em generalized pair} is a triple
$(X,B,\mathbf{M})$ consisting of a normal quasi-projective variety $X$,
an effective divisor $B$ on $X$,
and a b-nef divisor $\mathbf{M}$.
We further assume that $K_X+B+\mathbf{M}_X$ is an $\rr$-Cartier divisor. 
A generalized pair $(X,B,\mathbf{M})$ is said to be {\em $\mathbb{Q}$-factorial} if the underlying variety $X$ is $\mathbb{Q}$-factorial.

Let $(X,B,\mathbf{M})$ be a generalized pair and $\pi\colon Y\rightarrow X$ be a projective birational morphism from a normal variety.
Write $K_Y+B_Y+\mathbf{M}_Y=\pi^*(K_X+B+\mathbf{M}_X)$. 
If $B_Y\geq 0$, then we say that the generalized pair $(Y,B_Y,\mathbf{M})$ is induced by {\em log pull-back} of $(X,B,\mathbf{M})$ to $Y$.
}
\end{definition}

\begin{definition}
{\em 
Let $(X,B,\mathbf{M})$ be a generalized pair.
Let $\pi\colon Y\rightarrow X$ be a projective birational morphism
from a normal quasi-projective variety.
Let $E\subset Y$ be a prime divisor.
The {\em generalized log discrepancy}
of $(X,B,\mathbf{M})$ at $E$ 
is the number \[a_{E}(X,B,\mathbf{M})\coloneqq 1-{\rm coeff}_E(B_Y),\]
where $B_Y$ is uniquely determined by the following formula:
\[
K_Y+B_Y+\mathbf{M}_Y=\pi^*(K_X+B+\mathbf{M}_X).
\]
We say that the generalized pair
$(X,B,\mathbf{M})$ is:
\begin{enumerate}
\item {\em generalized terminal} if $a_{E}(X,B,\mathbf{M})>1$ for every exceptional $E$;
\item {\em generalized canonical} if $a_{E}(X,B,\mathbf{M}) \geq 1$ for every exceptional $E$;
\item {\em generalized Kawamata log terminal} (gklt) if $a_{E}(X,B,\mathbf{M}) > 0$ for every $E$;
\item {\em generalized log canonical} (glc) if $a_{E}(X,B,\mathbf{M}) \geq 0$ for every $E$.
\end{enumerate}
}
\end{definition}

\begin{definition}
{\em 
A {\em generalized log canonical place}
of a glc  
pair $(X,B,\mathbf{M})$ is a divisor $E$ over $X$ for which $a_E(X,B,\mathbf{M})=0$.
A {\em generalized log canonical center}
of a glc 
pair $(X,B,\mathbf{M})$ is the center on $X$ of a generalized log canonical 
place.
The {\em non-gklt locus} of $(X,B,\mathbf{M})$ is the union of all the generalized log canonical 
centers.
We may abbreviate generalized log canonical centers as glcc.

If the b-nef divisor $\mathbf{M}$ is trivial, then we drop the word {\em generalized} from the previous definitions, as we are in the usual setting of pairs.
}
\end{definition}

\begin{definition}
{\em 
A glc pair $(X,B,\mathbf{M})$ is said to be {\em generalized divisorially log terminal}
if there exists an open set $U\subseteq X$ satisfying the following conditions:
\begin{enumerate}
\item the variety $X$ is smooth on $U$ and $B$ is simple normal crossing on $U$,
\item the b-nef divisor $\mathbf{M}$ descends on $U$, 
\item the generic point of every glcc of $(X,B,\mathbf{M})$ is contained in $U$ and consists of a strata of $\lfloor B\rfloor$.
\end{enumerate}
For simplicity, we may say that $(X,B,\mathbf{M})$ is {\em gdlt}.

The {\em gdlt locus} of a generalized pair $(X,B,\mathbf{M})$ is the largest open subset $U\subset X$ for which the generalized pair $(U,B|_U,\mathbf{M}|_U)$ is gdlt. 
The {\em non-gdlt locus} of a generalized pair is the complement on $X$ of the gdlt locus.

Let $(X,B,\mathbf{M})$ be a glc pair.
A projective birational morphism
$\pi\colon Y \rightarrow X$ is said to be a {\em gdlt modification} if the following conditions are satisfied:
\begin{enumerate}
\item[(i)] for every $E\subset Y$ exceptional over $X$, we have $a_E(X,B,\mathbf{M})=0$, and 
\item[(ii)] the generalized pair $(Y,B_Y,\mathbf{M})$ obtained by log pull-back of $(X,B,\mathbf{M})$ to $Y$ is gdlt.
\end{enumerate}
If furthermore $Y$ is $\qq$-factorial,
then we say that $Y\rightarrow X$ is a {\em $\qq$-factorial gdlt modification}.
}
\end{definition}

The following lemma states that
every glc 
pair admits a $\qq$-factorial gdlt modification; see, e.g.,~\cite[Lemma 4.5]{BZ16} and ~\cite[Corollary 1.36]{Kol13}. 

\begin{lemma}\label{lem:gdlt}
The glc pair $(X,B,\mathbf{M})$ admits a $\qq$-factorial gdlt modification.
Furthermore, we may assume that the gdlt modification is an isomorphism
on the gdlt locus of $(X,B,\mathbf{M})$.
\end{lemma} 

The following lemma states that we can further assume that every finite set of lc places are divisorial in the chosen dlt modification.

\begin{lemma}\label{lem:special-gdlt}
Let $(X,B,\mathbf{M})$ be a glc 
pair.
Let $\mathcal{E}$ be a finite set of glc 
places of $(X,B,\mathbf{M})$.
Then, there exists a $\qq$-factorial gdlt modification $(Y,B_Y,\mathbf{M})$ of $(X,B,\mathbf{M})$ such that the center of every $F\in \mathcal{E}$ on $Y$ is a divisor $E_F$.
\end{lemma} 

\begin{proof} The proof of \cite[Corollary 1.38]{Kol13} works in the generalized setting too, see also \cite[Lemma 4.5]{BZ16}.
\end{proof}

\subsection{Fano and Calabi--Yau pairs}
In this section, we recall some concepts related to Fano varieties
and Calabi--Yau varieties.
We write some lemmata that will be used throughout the article.

\begin{definition}
{\em 
Let $X\rightarrow Z$ be a fibration.
We say that $X\rightarrow Z$ is {\em of Fano type}
if there exists a boundary $B$ on $X$ such that
$(X,B)$ is klt and $-(K_X+B)$ is ample over $Z$.
If $Z={\rm Spec}(\mathbb{K})$, then we say that the variety
$X$ is {\em of Fano type}.

Let $X\rightarrow Z$ be a fibration.
We say that $X\rightarrow Z$ is {\em of log Calabi--Yau type} if there exists a boundary $B$ on $X$ such that
$(X,B)$ is lc and $K_X+B$ is numerically trivial over $Z$.
If $Z={\rm Spec}(\mathbb{K})$, then we say that the variety $X$ is {\em of log Calabi--Yau type.}
}
\end{definition} 

Recall that a morphism of Fano type is a relative Mori dream space; see, e.g.,~\cite[Definition 3.13]{BM21}.
In particular, if $X\rightarrow Z$ is of Fano type, then the MMP over $Z$ of every divisor $D\subset X$ terminates with a good minimal model over $Z$
or a Mori fiber space over $Z$.

The following lemmas are well-known, 
we refer the reader to~\cite[Lemma 2.12]{Bir19} and \cite[Lemma 2.23]{Mor21}.

\begin{lemma}\label{lem:image-FT}
Let $X\rightarrow Z$ be a fibration over $S$.
If $X$ is of Fano type over $S$,
then $Z$ is of Fano type over $S$.
\end{lemma} 

\begin{lemma}\label{lem:contr-FT}
Let $(X, B, \mathrm{M})$ be a generalized log Calabi--Yau pair. Let $Y \dashrightarrow X$ be a birational map over $S$, only extracting divisors with log discrepancy in $[0,1)$ with respect to $(X, B, \mathrm{M})$.
If $X$ is of Fano type over $S$,
then $Y$ is of Fano type over $S$.
\end{lemma} 

The following lemma is proved using~\cite[Lemma 2.36]{Mor21}.

\begin{lemma}\label{lem:dlt-mod-base}
Let $\pi\colon X \rightarrow Z$ be a fibration of Fano type between $\qq$-factorial varieties. Let $(X,B,\mathbf{M})$ be a generalized log Calabi--Yau pair.
Let $(Z,B_Z,\mathbf{N})$ be the generalized pair obtained by the generalized canonical bundle formula.
Let $\phi_Z\colon Z'\rightarrow Z$ be a gdlt modification of $(Z,B_Z,\mathbf{N})$.
Then, there exists a commutative diagram
\[
\xymatrix{ 
(X,B,\mathbf{M})\ar[d]_-{\pi} & (X',B',\mathbf{M})\ar@{-->}[l]_-{\phi}\ar[d]^-{\pi'} \\ 
(Z,B_Z,\mathbf{N}) &
(Z',B_{Z'},\mathbf{N})\ar[l]^-{\phi_Z}
}
\]
satisfying the following conditions:
\begin{enumerate}
\item $X'$ is $\qq$-factorial,
\item $\phi$ is a crepant birational map
only extracting glc places, 
\item $\phi$ is an isomorphism at the generic point of $Z$, and
\item ${\pi'}^{-1}(\lfloor B_{Z'}\rfloor)\subseteq \lfloor B'\rfloor$.
\end{enumerate}
Further, if $\pi$ is a Mori fiber space, then $\pi'$ is a Mori fiber space.
\end{lemma}

\begin{proof}
By~\cite[Lemma 2.36]{Mor21}, there is a commutative diagram
\[
\xymatrix{ 
(X,B,\mathbf{M})\ar[d]_-{\pi} & (X_0,B_0,\mathbf{M})\ar@{-->}[l]_-{\phi}\ar[d]^-{\pi'} \\ 
(Z,B_Z,\mathbf{N}) &
(Z',B_{Z'},\mathbf{N})\ar[l]^-{\phi_Z}
}
\]
satisfying $(1)$, $(2)$, $(4)$, and that $X_0\rightarrow Z'$ is a fibration of Fano type.
Let $A$ be an ample divisor on $X$ and $A_0$ be its strict transform on $X_0$. We replace $X_0$ by a small $\qq$-factorialization of the relative ample model of $A_0$ over $Z'$. By doing so, condition $(3)$ is achieved, while $(1), (2),$ and $(4)$ are preserved.

Suppose further that $\pi$ is a Mori fiber space.
By $(3)$, the fibration $X_0\rightarrow Z'$ has relative Picard rank one over the complement of $\lfloor B_{Z'}\rfloor$.
By $(4)$, we can find a reduced divisor $S_0\subseteq \lfloor B_0\rfloor$ such that 
${\pi'}^{-1}(Q)$ contains a unique prime component of $S_0$ for each prime component $Q$ of $\lfloor B_{Z'}\rfloor$.
By running a $(\lfloor B_{0}\rfloor-S_0)$-MMP
 over $Z'$, we get a crepant birational model $(X',B',\mathbf{M})$ of $(X_0,B_0,\mathbf{M})$ with a fibration $\pi'\colon X'\rightarrow Z'$ of Fano type.
By~\cite[Lemma 2.9]{Lai11}, the previous MMP contracts all the prime components of $\lfloor B_{0}\rfloor-S_0$.
Conditions $(1)$-$(4)$ are preserved by this MMP, so $(X',B',\mathbf{M})$ satisfies these conditions.
Furthermore, the preimage of any
irreducible component of $\lfloor B_{Z'}\rfloor$ via $\pi'$ is irreducible.
Thus, $\pi'$ has relative Picard rank one and $(4)$ is achieved.
\end{proof}

The following is a consequence of~\cite[Lemma 3.3 \& 3.4]{Mor21}.

\begin{lemma}\label{lem:two-ray-game}
Let $\pi\colon X \rightarrow Z$ be a fibration between $\qq$-factorial varieties of Fano type over $\schbase$.
Let $Z\dashrightarrow Z'$ be a birational contraction over $\schbase$
with $Z'$ being $\qq$-factorial.
Then, there exists a commutative diagram over $\schbase$
\[
\xymatrix{
X\ar[d]_-{\pi}\ar@{-->}[r]^-{\phi} & X' \ar[d]^-{\pi'} \\
Z\ar@{-->}[r]_-{\phi_Z} & Z' \\
}
\]
where $\phi$ is a birational contraction and an isomorphism at the generic point of $Z$, 
$X'$ is $\qq$-factorial, and $\pi'$ is a fibration of Fano type. 
Further, if $\pi$ is a Mori fiber space, then $\pi'$ is a Mori fiber space.
\end{lemma}

\begin{proof}
Since $X$ is of Fano type over $\schbase$,
then $Z$ is of Fano type over $\schbase$ as well; see Lemma~\ref{lem:image-FT}.
Then, we can find boundaries
$\Delta$ and $\Delta_Z$ on $X$ and $Z$, respectively, 
that are big over $S$,
such that
$(X,\Delta)$ and $(Z,\Delta_Z)$ are klt log Calabi--Yau.
Let $(Z',\Delta_{Z'})$ be the log Calabi--Yau pair induced on $Z'$.
Let $A_{Z'}$ be an ample divisor on $Z'$
and $A_Z$ be its strict transform on $Z$.
If $\epsilon>0$ is small enough, then $Z'$ is an ample model
for the klt pair
$(Z,\Delta_Z+\epsilon A_Z)$.
Let $Z\dashrightarrow Z_1\dashrightarrow \dots \dashrightarrow Z_k$ be the steps
of a $(K_Z+\Delta_Z+\epsilon A_Z)$-MMP over $S$.
Since $Z_k$ is a minimal model for $A_Z$,
we conclude that there is a projective birational morphism $p\colon Z_k\rightarrow Z'$
between $\qq$-factorial varieties of Fano type.
Let $E$ be the exceptional divisor of $p$.
The $E$-MMP over $Z$ terminates on $Z$ itself.
Let $Z_k\dashrightarrow Z_{k+1}\dashrightarrow \dots \dashrightarrow Z_r:=Z$ be the steps of this MMP.

By \cite[Proof of Lemma 3.3 \& 3.4]{Mor21},
there exists a commutative diagram 
\[
\xymatrix{
X\ar[d]^-{\pi}\ar@{-->}[r]^-{\phi_1} & X_1\ar[d]^-{\pi_1}\ar@{-->}[r]^-{\phi_2} & X_2\ar[d]^-{\pi_3}\ar@{-->}[r]^-{\phi_1} & \dots \ar@{-->}[r]^-{\phi_r} & X_r\ar[d]^-{\pi_k} \\
Z \ar@{-->}[r] & Z_1 \ar@{-->}[r] & Z_2 \ar@{-->}[r] & \dots \ar@{-->}[r] & Z_r,
}
\]
where each $\phi_k$ is a 
birational contraction between $\qq$-factorial varieties.
We set $X':=X_r$ and $\pi':=\pi_r$.
The birational map $\phi_i$ is small whenever
$Z_{i-1}\dashrightarrow Z_i$ is small,
and $\phi_i$ contracts a prime divisor 
whenever $Z_{i-1}\rightarrow Z_i$ contracts a prime divisor.
Hence, $\rho(\pi_i)$ is independent of $i$.
This finishes the proof.
\end{proof} 

\begin{lemma}
\label{cor:mod-base-mod-tot-space}
Let $\pi\colon X \rightarrow Z$ be a fibration between $\qq$-factorial varieties of Fano type over $\schbase$. Let $(X,B,\mathbf{M})$ be a generalized log Calabi--Yau pair.
Let $(Z,B_Z,\mathbf{N})$ be the generalized pair obtained by the generalized canonical bundle formula. Let $\phi_Z\colon Z'\dashrightarrow Z$ be a birational map
between $\qq$-factorial varieties extracting only glc places of $(Z,B_Z,\mathbf{N})$.
Then, there exists a commutative diagram
\[
\xymatrix{ 
(X,B,\mathbf{M})\ar[d]_-{\pi} & (X',B',\mathbf{M})\ar@{-->}[l]_-{\phi}\ar[d]^-{\pi'} \\ 
(Z,B_Z,\mathbf{N}) &
(Z',B_{Z'},\mathbf{N})\ar@{-->}[l]^-{\phi_Z}
}
\]
satisfying the following conditions:
\begin{enumerate}
\item $X'$ is $\qq$-factorial,
\item $\phi$ is a crepant birational map
only extracting glc places, 
\item $\phi$ is an isomorphism at the generic point of $Z$, and
\item ${\pi'}^{-1}(\lfloor B_{Z'}\rfloor)\subseteq \lfloor B'\rfloor$.
\end{enumerate}
Further, if $\pi$ is a Mori fiber space, then $\pi'$ is a Mori fiber space.
\end{lemma}
\begin{proof}
The birational map $\phi_{Z}$ factors as 
\[\phi_{Z} \colon (Z',B_{Z'},\mathbf{N}) \dashleftarrow (Z_0,B_{Z_0},\mathbf{N}) \to (Z,B_Z,\mathbf{N})\]
where $(Z_0,B_{Z_0},\mathbf{N})$ is a $\qq$-factorial dlt modification of $(Z,B_Z,\mathbf{N})$ extracting the glc places contracted by $\phi_{Z}$, which exists by Lemma~\ref{lem:special-gdlt}. By assumption, all the divisors contracted by $\phi_{Z}$ are glc places, so $(Z_0,B_{Z_0},\mathbf{N}) \dashrightarrow (Z',B_{Z'},\mathbf{N})$ must be a birational contraction. The statement now follows from Lemma~\ref{lem:dlt-mod-base} and Lemma~\ref{lem:two-ray-game}.
\end{proof}

The following lemma is well-known to the experts. The proof in the case of pairs is in~\cite[Theorem 3.1]{BDCS20}.
In the case of generalized pairs, the proof works verbatim as in~\cite[Proof of Proposition 3.2]{Mor21}.

\begin{lemma}\label{lem:tw-mfs}
Let $(X,B,\mathbf{M})$ be a gklt Calabi--Yau pair.
Then, there exists a birational contraction $X\dashrightarrow X'$
and a fibration $X'\rightarrow Z$ satisfying the following:
\begin{enumerate}
\item the fibration $X'\rightarrow Z$ is a composition of Mori fiber spaces, and 
\item either $Z$ is a point or a klt variety with $K_Z\equiv 0$.
\end{enumerate} 
Furthermore, both $X'$ and $Z$ can be chosen to be $\qq$-factorial. 
\end{lemma}

As a consequence, we have the following lemma.

\begin{lemma}\label{lem:tower-mfs-picard}
Let $X$ be a variety of 
Fano type, and $(X,B,\mathbf{M})$ be a gklt Calabi--Yau pair.
Then, there exists a birational contraction $X\dashrightarrow X'$ with $\rho(X')\leq \dim(X')$. Furthermore, $X'$ can be chosen to be $\qq$-factorial.
\end{lemma} 

\begin{proof}
By Lemma~\ref{lem:tw-mfs}
there is a birational contraction $X\dashrightarrow X'$ and a fibration $\phi \colon X'\rightarrow X_k$
that factorizes as
$X_0:=X' \rightarrow X_1 \rightarrow \dots \rightarrow X_k$, where each
$\phi_k \colon X_i\rightarrow X_{i+1}$ is a Mori fiber space.
Furthermore, either 
$\dim X_k=0$ or $X_k$ is klt Calabi--Yau with $K_{X_k}\equiv 0$. 
We argue that this latter case does not happen.
Indeed, by construction $X'$ is of Fano type
and hence by Lemma~\ref{lem:image-FT} each $X_i$, including $X_k$, must be of Fano type.
If $X_k$ is of Fano type, then $K_{X_k}\not \equiv 0$.
We conclude that $\dim X_k=0$.
Since $\dim(X_i) \geq \dim(X_{i+1})+1$ and $\rho(X_i) = \rho(X_{i+1})+1$, we obtain that $\rho(X_i) \leq \dim(X_{i})$.
\end{proof} 

The following Lemma \ref{lem:from-FT-to-tw-MFS} is a relative version of Lemma \ref{lem:tw-mfs}.

\begin{lemma}\label{lem:from-FT-to-tw-MFS}
Let $X\rightarrow Z$ be a morphism of Fano type.
Then, there exists a birational contraction $X\dashrightarrow X'$ over $Z$ and a fibration 
$X'\rightarrow Z'$ over $Z$ such that
\begin{enumerate}
\item the morphism $Z'\rightarrow Z$ is birational, and 
\item the fibration $X'\rightarrow Z'$ is a tower of Mori fiber spaces.
\end{enumerate}
Furthermore, both $X'$ and $Z'$ can be chosen to be $\qq$-factorial.
\end{lemma}

\begin{proof}
We run a $K_X$-MMP over $Z$ and let
$X\dashrightarrow X_1 \dashrightarrow  \dots \dashrightarrow X_k$ be the steps of this MMP and $X_k\rightarrow W$ be the associated Mori fiber space. 
If $W\rightarrow Z$ is birational, then we are done.
Indeed, we just set $Z':=W$.
Assume that $W\rightarrow Z$ is not birational.
By Lemma~\ref{lem:image-FT}, we know that $W\rightarrow Z$ is a morphism of Fano type.
Hence, we may proceed by induction on the dimension.
There exists a birational contraction
$W\dashrightarrow W_m$ and a fibration $W_m\rightarrow Z'$ over $Z$ such that:
\begin{itemize}
\item the fibration $W_m\rightarrow Z'$ is composition of Mori fiber spaces, and 
\item the morphism $Z'\rightarrow Z$ is birational.
\end{itemize} 
Applying Lemma~\ref{lem:two-ray-game}, there exists a birational contraction $X_k\dashrightarrow X_m$
and a Mori fiber space $X_m\rightarrow W_m$
making the following diagram commutative:
\[
\xymatrix{
X\ar[dd]\ar@{-->}[r] & X_k\ar@{-->}[r]\ar[d] & X_m\ar[d] \\
& W
\ar@{-->}[r]\ar[ld] & W_m\ar[d] \\ 
Z & & Z'.\ar[ll]
}
\]
Note that $X_m\rightarrow Z'$ is a composition of Mori fiber spaces. Further, the outcome of an MMP and the base of Mori fiber spaces are $\qq$-factorial whenever the domain is $\qq$-factorial, which can be granted up to a small $\qq$-factorialization of $X$. This implies that both $X'$ and $Z'$ can be chosen $\qq$-factorial, and finishes the proof.
\end{proof} 

The following Lemma \ref{lem:bounding-comp} is a generalization of Lemma \ref{lem:tower-mfs-picard} for glc pairs.

\begin{lemma}\label{lem:bounding-comp}
Let $(X,B,\mathbf{M})$ be a generalized log Calabi--Yau pair of dimension $n$. 
There exists a crepant model $(X',B',\mathbf{M})$ of $(X,B,\mathbf{M})$ and a fibration $X'\rightarrow Z$ satisfying the following:
\begin{enumerate}
\item the variety $X'$ is $\qq$-factorial, 
\item every divisor extracted by $X'\dashrightarrow X$ is a glc place of $(X,B,\mathbf{M})$, 
\item every component of $B'$ dominates $Z$, and
\item $X' \to Z$ is a tower of Mori fiber spaces, so $\rho(X'/Z) \leq n$. 
\end{enumerate}
In particular, $Z$ is either a point or a klt variety with $K_{Z} \equiv 0$, and $X'\rightarrow Z$ is the MRC fibration of $X$.
\end{lemma} 

\begin{proof}
By~\cite[Theorem 49]{KX16} and~\cite[Theorem 4.2]{FS20},
 there exists a crepant model $(X_0,B_0,\mathbf{M})$ of $(X,B,\mathbf{M})$ and a fibration $X_0\rightarrow Z_0$ satisfying the following conditions:
\begin{itemize}
\item $X_0$ is $\qq$-factorial,
\item every exceptional divisor of $X_0\dashrightarrow X$ is a glc place of $(X,B,\mathbf{M})$,
\item every glcc of $(X_0,B_0,\mathbf{M})$ dominates $Z_0$, and 
\item $X_0$ is of Fano type over $Z_0$.
\end{itemize} 
We may assume that $Z_0$ is $\qq$-factorial.
Let $(Z_0,B_{Z_0},\mathbf{N})$ be the generalized Calabi--Yau pair induced by the generalized canonical bundle formula.
Then, $(Z_0,B_{Z_0},\mathbf{N})$ is gklt. By Lemma~\ref{lem:tw-mfs}, there
exists a birational contraction $Z_0\dashrightarrow Z_1$ 
and a fibration $Z_1\rightarrow Z$ satisfying the following:
\begin{itemize}
\item $Z_1$ and $Z$ are $\qq$-factorial,
\item $Z_1\rightarrow Z$ is a composition of Mori fiber spaces, and 
\item either $Z$ is a point or a klt variety with $K_Z\equiv 0$.
\end{itemize} 
Apply Lemma~\ref{lem:two-ray-game} recursively 
to obtain a birational contraction 
$X_0\dashrightarrow X_1$
to a $\qq$-factorial variety $X_1$ and a fibration $X_1\rightarrow Z_1$ of Fano type.
By Lemma~\ref{lem:from-FT-to-tw-MFS}, 
there exist a birational contraction
$X_1\dashrightarrow X_2$,
a composition of Mori fiber spaces $ X_2\rightarrow Z_2$, and a birational morphism $Z_2\rightarrow Z_1$, such that $X_2$ and $Z_2$ are $\qq$-factorial, and
the following diagram commutes:
\[
\xymatrix{
X & X_0\ar@{-->}[l]\ar@{-->}[r]\ar[d] & X_1\ar[d]\ar@{-->}[r] & X_2\ar[d] \\ 
& Z_0\ar@{-->}[r] & Z_1\ar[d] & Z_2\ar[l] \\ 
& &  Z. & 
}
\]
Let $E$ be the reduced exceptional divisor of 
$Z_2\rightarrow Z_1$. The $E$-MMP over $Z_1$ contracts all the components of $E$, and terminates with $Z_1$, since $Z_1$ is $\qq$-factorial. By Lemma \ref{lem:two-ray-game}, there exist a birational contraction 
$X_2\dashrightarrow X_3$,
a composition of Mori fiber spaces $X_3\rightarrow Z_3$, such that $X_3$ is $\qq$-factorial, and the following diagram commutes
\[
\xymatrix{
X & X_0 \ar@{-->}[l] \ar@{-->}[r] & X_2\ar@{-->}[r]\ar[d] & X_3\ar[d] \\ 
& & Z_2\ar@{-->}[r] \ar[d]& Z_1\ar[ld]  \\ 
& & Z. & 
}
\]

We set $X':=X_3$.
Let $(X',B',\mathbf{M})$ be the glc pair induced on $X'$.
By construction, we know that $X'$ is $\qq$-factorial so $(1)$ holds. 
As $X_0\dashrightarrow X$ only extracts glc places of $(X,B,\mathbf{M})$ and $X_0\dashrightarrow X_3$ is a contraction, then $(2)$ holds.
Since $Z$ is klt Calabi--Yau, every component of $B'$ dominates $Z$. By construction, $X' \to Z$ is a tower of Mori fiber space. This finishes the proof.

\end{proof}

In the following lemmas, we study the birational behaviour of ample divisors supported on the boundary of a log Calabi--Yau pair.

\begin{lemma}\label{lem:dlt-mod-supp-amp}
Let $(X,B,\mathbf{M})$ be a $\qq$-factorial generalized log Calabi--Yau pair. 
Let $X\rightarrow Z$ be a fibration.
Let $E$ be an effective divisor satisfying the following conditions:
\begin{enumerate}
\item $E$ is supported on $\lfloor B\rfloor$, 
\item $E$ supports an ample divisor over $Z$, and 
\item $E$ contains every non-gdlt center of $(X,B,\mathbf{M})$.
\end{enumerate} 
There exists a $\qq$-factorial gdlt modification $p \colon (X',B',\mathbf{M})\to 
(X,B,\mathbf{M})$ such that:
\begin{enumerate}
\item[(i)] $p^*E$ is supported on $\lfloor B'\rfloor$, and
\item[(ii)] $p^*E$ supports an ample divisor over $Z$.
\end{enumerate} 
\end{lemma}

\begin{proof}
Let $p\colon X'\rightarrow X$ be a $\qq$-factorial gdlt modification of $(X,B,\mathbf{M})$ that is an isomorphism over the gdlt locus of $(X,B,\mathbf{M})$; see Lemma~\ref{lem:gdlt}.
Since $X$ is $\qq$-factorial, there exists an effective $p$-exceptional divisor $F$ on $X'$ such that $-F$ is $p$-ample.
On the other hand, since the support of $E$ 
contains every non-gdlt center of $(X,B,\mathbf{M})$, we have that the support of $p^*E$ contains the support of $F$. In particular, we have that 
$p^*E-\epsilon F$ is ample for $\epsilon>0$ small enough.
It suffices to show that $p^*E$ is supported on $\lfloor B'\rfloor$. 
Note that $p^*E$ is supported on the union
of the strict transform of $\lfloor B\rfloor$ and the $p$-exceptional divisors.
Since $X'\rightarrow X$ is a gdlt modification, then every exceptional divisor is contained in $\lfloor B'\rfloor$. This finishes the proof. 
\end{proof} 

\begin{lemma}\label{lem:-S-MMP-non-dlt}
Let $(X,B,\mathbf{M})$ be a $\qq$-factorial gdlt log Calabi--Yau pair.
Let $S$ be a component of $\lfloor B\rfloor$.
Then, we may run a $(K_X+B-S+\mathbf{M}_X)$-MMP 
with steps 
$X\dashrightarrow X_1 \dashrightarrow X_2 \dashrightarrow \dots \dashrightarrow X_k$.
For each $i$, we let $(X_i,B_i,\mathbf{M})$ be the induced generalized log Calabi--Yau pair on $X_i$ and $S_i$ the strict transform of $S$ on $X_i$.
Then, for each $i$, the non-gdlt locus of 
$(X_i,B_i,\mathbf{M})$ is contained in the support of $S_i$.
\end{lemma} 

\begin{proof}
First, note that we can run a $(K_X+B-S+\mathbf{M}_X)$-MMP by~\cite[Lemma 4.4]{BZ16}.
Gdlt singularities are preserved under the MMP. Hence, for each $i$, the generalized pair
$(X_i,B_i-S_i,\mathbf{M})$ is gdlt.  
This implies that $(X_i,B_i,\mathbf{M})$ is gdlt on $X_i\setminus S_i$.
Thus, any non-gdlt center must be contained in $S_i$.
\end{proof} 

\begin{lemma}\label{lem:prop-supp-ampl-mod}
Let $(X,B,\mathbf{M})$ be a $\qq$-factorial gdlt log Calabi--Yau pair over $W$.
Assume that $\lfloor B\rfloor$ properly supports an effective divisor that is ample over $W$.
Let $p \colon Y\rightarrow X$ be a projective birational morphism extracting only glc places of $(X,B,\mathbf{M})$.
Let
$(Y,B_Y,\mathbf{M})$ be the generalized pair induced on $Y$ by log pull-back.
Then $\lfloor B_Y\rfloor$ properly supports an effective divisor that is ample over $W$.
\end{lemma}

\begin{proof}
Let $A$ be the effective ample divisor that
is properly supported on $\lfloor B\rfloor$.
Let $S$ be a component of $\lfloor B\rfloor$
that is not contained in the support of $A$.
Then, for $\epsilon>0$ small enough, the divisor
$A':=A+\epsilon(\lfloor B\rfloor-S)$ is effective, 
ample over $W$, and properly supported on $\lfloor B\rfloor$.
Furthermore, the support of $A'$ 
 contains every gdlt center of $(X,B,\mathbf{M})$ except for $S$.
Since $X$ is $\qq$-factorial, there exists an effective $p$-exceptional divisor $F$
such that $-F$ is $p$-ample.
The center of any $p$-exceptional divisor in $X$ is a non-divisorial gdlt center of $(X,B,\mathbf{M})$, so it is contained in $A'$. Equivalently, $p^*A'$
contains $F$. 
Hence, the divisor $p^*A'-\delta F$ is effective for $\delta$ small enough and it is ample over $W$.
Note that $p^*A'-\delta F$ is contained in $\lfloor B_Y\rfloor$.
Furthermore, the support of $p^*A'-\delta F$ does not contain $S_Y\subset \lfloor B_Y\rfloor$, the strict transform of $S$ in $Y$.
We conclude that $p^*A'-\delta F$ is an effective divisor, ample over $W$, and properly supported in $\lfloor B_Y\rfloor$.
\end{proof}

\subsection{Dual complexes}
In this section, we recall the concept of dual complexes of Calabi--Yau pairs (or generalized pairs). For cell complexes, we follow the terminology of \cite[\S 2.1]{Hatcher}.

\begin{definition}\label{defn:dualcomplex}
{\em 
Let $(Y, B_Y)$ be a dlt pair with $B^{=1}_Y=\sum_{i \in I} B_{i}$. 
The \emph{dual complex} of $(Y, B_Y)$, denoted $\mathcal{D}(B_Y)$, is the regular $\Delta$-complex whose vertices correspond to irreducible components of $B^{=1}_Y$
and whose $k$-cells correspond to strata of $B^{=1}_Y$, i.e., connected components of $B_{i_0}\cap \dots \cap B_{i_{k}}$.
}
\end{definition}

\begin{definition}
Let $(X,B,\mathbf{M})$ be a generalized log Calabi--Yau pair (not necessarily dlt). Then $\mathcal{D}(X,B,\mathbf{M})$ is the PL-homeomorphism type of the dual complex of a gdlt modification $(Y,B_Y,\mathbf{M}) \to (X,B,\mathbf{M})$, i.e.,
$\mathcal{D}(X,B,\mathbf{M}) \simeq_{PL}\mathcal{D}(B_{Y})$.
\end{definition}

The definition is independent of the choice of the dlt modification by \cite[Proposition 11]{dFKX17}. We recall the following statement from~\cite[Theorem 19]{dFKX17}.

\begin{proposition}\label{prop:dc-collapse}
Let $(X,\Delta)$ be a dlt pair.
Let $X\dashrightarrow X'$ be either a flip
or a divisorial contraction induced by a $(K_X+\Delta)$-negative extremal ray $R$.
Let $\Delta'$ be the image of $\Delta$ on $X'$.
Assume that there exists a prime component $S$ of $\lfloor \Delta \rfloor$ 
such that $S \cdot R>0$. Then $\mathcal{D}(X,\Delta)$ collapses to $\mathcal{D}(X',\Delta')$.
\end{proposition}

\subsection{Coregularity}
The coregularity of a log Calabi--Yau pair is an invariant measuring the dimension of its minimal strata. It appears in previous work in relation to the indices of log Calabi--Yau pairs~\cite{FMM22}, or in connections with 
log canonical thresholds and existence of complements~\cite{FMP22,FFMP22,Mor22}.

\begin{definition}
{\em Let $(X,B,\mathbf{M})$ be a generalized log Calabi--Yau pair. 
The {\em coregularity} of $(X,B,\mathbf{M})$, denoted ${\rm coreg}(X,B,\mathbf{M})$, is the non-negative integer
\[
\dim X - \dim \mathcal{D}(X,B,\mathbf{M}) -1. 
\]
The coregularity of a log Calabi--Yau pair of dimension $n$ takes values in the set $\{0,1,\dots,n\}$. 
Equivalently, it is the dimension of a minimal lc center
in any gdlt modification of a log Calabi--Yau pair.
}
\end{definition} 

The following lemma shows that the coregularity is invariant for crepant birational log Calabi--Yau pairs; see, e.g.,~\cite[Proposition 3.11]{FMM22}.

\begin{lemma}\label{lem:coreg-bir}
Let $(X,B,\mathbf{M})$ and $(X',B',\mathbf{M})$ be a crepant birational generalized log Calabi--Yau pair.
Then, we have that
\[
{\rm coreg}(X,B,\mathbf{M})=
{\rm coreg}(X',B',\mathbf{M}).
\]
\end{lemma} 

\subsection{Complexity}
The complexity of a generalized pair
$(X,B,\mathbf{M})$ is an invariant that detects whether the pair is toric. 

\begin{definition}
{\em 
Let $(X,B,\mathbf{M})$ be a generalized log pair with $B=\sum_{i \in I} a_i B_i$. Set $|B| \coloneqq \sum_{i \in I} a_i$, and $\langle B \rangle \subset \Cl_\qq(X)$ be the span of the Weil divisors $B_i$. 
The {\em complexity} of the generalized pair $(X,B,\mathbf{M})$
is the number 
\[
c(X,B,\mathbf{M}):= \dim X + \rank \Cl_\qq(X) - |B|.
\]
The {\em fine complexity} of the generalized pair $(X,B,\mathbf{M})$
is the number 
\[
\overline{c}(X,B,\mathbf{M}):= \dim X + \rank \langle B \rangle - |B|.
\]
}
\end{definition} 

The following theorem is proved in~\cite[Theorem 1.2]{BMSZ18}; see also~\cite{MS21}. It states that the complexity
of a generalized log Calabi--Yau pair is non-negative.
Further, if the complexity is zero, then 
the ambient variety $X$ is toric and some of the components of $B$ are torus invariant.

\begin{theorem}\label{thm:comp-noneg}
Let $(X,B,\mathbf{M})$ be a generalized log Calabi--Yau pair.
Then, we have that
\[c(X,B,\mathbf{M})\geq \overline{c}(X,B,\mathbf{M})\geq 0.\]
Furthermore, if the equality holds 
then $\mathbf{M}$ is the trivial b-nef divisor,  $(X, \lfloor B \rfloor)$ is a toric pair. 
\end{theorem}

\subsection{Birational complexity}
The birational complexity is the mimimum
among all the complexities 
of birational models of a given generalized pair.

\begin{definition}
{\em 
Let $(X,B,\mathbf{M})$ be a generalized log Calabi--Yau pair.
We say that a generalized pair $(X',B',\mathbf{M})$ is a {\em crepant model} of $(X,B,\mathbf{M})$ if the following conditions are satisfied:
\begin{itemize}
\item $X'$ is birational to $X$, and
\item there exists a common resolution $p\colon Y\rightarrow X$ and $q\colon Y\rightarrow X'$ for which 
\[
p^*(K_X+B+\mathbf{M}_X)=q^*(K_{X'}+B'+\mathbf{M}_{X'}).
\]
\end{itemize} 
}
\end{definition} 

From the definition it follows that any crepant model of a generalized log Calabi--Yau pair is generalized log Calabi--Yau. 

\begin{definition}
{\em 
The {\em birational complexity} of a generalized log Calabi--Yau pair $(X,B,\mathbf{M})$ is the infimum:
\[
c_{\rm bir}(X,B,\mathbf{M}):=
\inf \{ c(X',B',\mathbf{M}) \mid\text{$(X',B',\mathbf{M})$ is a crepant model of $(X,B,\mathbf{M})$}\}.
\]
}
\end{definition} 

The definition of the birational complexity
and the fact that the complexity is nonnegative
implies the following:

\begin{proposition}\label{cor:bcomp-noneg}
Let $(X,B,\mathbf{M})$ be a generalized log Calabi--Yau pair.
Then, 
$c_{\rm bir}(X,B,\mathbf{M})\geq 0$.
\end{proposition} 

The following three lemmas imply that the birational complexity is indeed a minimum,
and it is computed by some $\qq$-factorial gdlt pair.
\begin{lemma}\label{lem:bcomp-dlt-mod}
Let $(X,B,\mathbf{M})$ be a generalized log Calabi--Yau pair. Let $Y \dashrightarrow X$ be a $\qq$-factorial birational contraction extracting only glc places of $(X,B,\mathbf{M})$, e.g., a dlt modification of $(X,B,\mathbf{M})$. Let $(Y,B_Y,\mathbf{M})$ be the generalized pair induced by log pullack on $Y$. Then, we have that
\begin{align*}
|B_X|-\rho(X) & \geq |B_Y|-\rho(Y),\\
|\lfloor B_{X}\rfloor|-\rho(X) & \geq |\lfloor B_{Y}\rfloor|-\rho(Y), \\ c(X,B,\mathbf{M}) & \geq c(Y,B_Y,\mathbf{M}),
\end{align*}
and equality holds if and only if $X$ is $\qq$-factorial.
\end{lemma} 

\begin{proof}
Let $e$ be the number of prime exceptional divisors of $Y\dashrightarrow X$. Since every exceptional divisor appears in $\lfloor B_Y\rfloor$, we have that \[|B_Y|=|B|+e, \qquad |\lfloor B_Y\rfloor|=|\lfloor B\rfloor|+e, \qquad \rho(Y)-\rho(X) \geq \dim_{\qq} \Cl(Y)_{\qq} - \dim_{\qq} \Cl(X)_{\qq} = e,\] and the equality holds if and only if $X$ is $\qq$-factorial. Hence, the inequalities of the statement follows immediately. 
\end{proof}

\begin{lemma}\label{lem:computing-bcomp}
Let $(X,B,\mathbf{M})$ be a generalized log Calabi--Yau pair.
A crepant model computing the birational complexity of $(X,B,\mathbf{M})$ exists, and it can be chosen to be $\qq$-factorial and gdlt. 
\end{lemma}

\begin{proof}
By Lemma~\ref{lem:bounding-comp}, 
there exists a $\qq$-factorial crepant model $(X',B',\mathbf{M})$ of $(X,B,\mathbf{M})$
that admits a fibration $X'\rightarrow Z$ satisfying the following conditions:
\begin{enumerate}
\item \label{item:glcplace} every divisor extracted by $X'\dashrightarrow X$ is a glc place of $(X,B,\mathbf{M})$, 
\item \label{item:dominate} every component of $B'$ dominates $Z$, and 
\item \label{item:Picbound} $X' \to Z$ is a tower of Mori fiber spaces, so $\rho(X'/Z) \leq n$. 
\end{enumerate}
By \eqref{item:glcplace} and Lemma \ref{lem:bcomp-dlt-mod}, we can write
\begin{equation}\label{eq:ineq-comple}
c(X,B,\mathbf{M})\geq c(X',B',\mathbf{M}) = \dim X' + \rho(X') - |B'|. 
\end{equation}
We claim that the RHS of \eqref{eq:ineq-comple} attains value in a discrete set of non-negative numbers, so it admits a minimum. To this end, it suffices to show that $|B'|$ lies in a finite set independent of the crepant model $(X',B',\mathbf{M})$. Indeed, by~\cite[Corollary 2.13]{KM98}
applied to a log resolution of $(X,B,\mathbf{M})$,
there is a finite set $\mathcal{B}\subset [0,1]$ such that every non-terminal valuation $F$ over $(X,B,\mathbf{M})$ satisfies that
$a_F(X,B,\mathbf{M})\in \mathcal{B}$. Set $\mathcal{B}':=\{1-b\mid b\in \mathcal{B}\}$ and $\gamma \coloneqq \min \{ \mathcal{B}'\}$. Then the coefficients of $B'$ are contained in $\mathcal{B}'$, and by the non-negativity of the complexity in the relative setting (see \cite[Theorem 1]{MS21}), and since $\rho(X'/Z) \leq n$, $B'$ must have at most $k:=\lfloor \gamma^{-1}2n\rfloor$ components. Note that $k$ is independent of the chosen crepant model $(X',B',\mathbf{M})$. Therefore, the coefficients of $B'$, and so $|B'|$, belongs to a finite set independent of the crepant model $(X',B',\mathbf{M})$.

We conclude that the birational complexity is computed by a crepant model, and by Lemma \ref{lem:bcomp-dlt-mod}, the model can be chosen to be $\qq$-factorial and gdlt.
\end{proof} 

\begin{lemma}\label{lem:model-comp-bcomp}
Let $(X,B,\mathbf{M})$ be a $\qq$-factorial 
generalized log Calabi--Yau pair for which $c_{\rm bir}(X,B,\mathbf{M})=c(X,B,\mathbf{M})$.
Let $q\colon X\dashrightarrow X'$ be a birational contraction onto a $\qq$-factorial variety.
Then, the divisorial part of the exceptional locus of $q$ is contained in $\lfloor B\rfloor$.
Furthermore, if $(X',B',\mathbf{M})$ is the generalized pair induced on $X'$,
then 
$c_{\rm bir}(X',B',\mathbf{M})=
c(X',B',\mathbf{M})$.
\end{lemma} 

\begin{proof}
Let $r$ be the number of divisors contracted by $q$,
and $s$ be the sum of their coefficients in $B$.
Then, we have that
\[
c(X,B,\mathbf{M}) = c(X',B',\mathbf{M}) + r - s. 
\]
As we are assuming that $(X,B,\mathbf{M})$ computes the birational complexity, we must have $r=s$, i.e., every exceptional divisor is a component of $\lfloor B\rfloor$. 
\end{proof} 

\subsection{PL spheres}
In this section, we discuss a characterization of PL spheres. We refer to \cite{RS1982} for the standard terminology in piecewise linear topology.

\begin{proposition}\label{prop:PL-sphere}
Let $M$ be an $n$-dimensional closed PL manifold.
Assume that $M$ is the union of two collapsible subcomplexes $D_1$ and $D_2$ glued along their boundaries $\partial D_1 = \partial D_2$, i.e., $M=D_1 \cup_{\partial D_1 = \partial D_2} D_2$.

If $D_1 \simeq_{\rm PL} B^{n}$,
then $M \simeq S^{n}$, and $M\simeq_{\rm PL} S^n$ if $n\neq 4$.
\end{proposition}

\begin{proof}
A regular neighborhood $N$ of a subcomplex $X \subset M$ is a neighbourhood of $X$, collapsing onto $X$, which is further a PL-submanifold of $M$ with boundary; see \cite[Corollary 3.30]{RS1982}. By \cite[Corollary 3.27]{RS1982}, any regular neighborhood $N_1$ of $D_1$ is PL-homeomorphic to $B^{n}$. Moreover, $N_1 \setminus D_1$ is homotopic equivalent to $S^{n-1}$. This follows from \cite[Corollary 3.19]{RS1982} together with the fact that $D_1$ is a regular neighbourhood of itself. Now suppose without loss of generality that $n>1$. Apply Mayer--Vietoris and Van Kampen theorems to the contractible open sets $\mathrm{int}(N_1)$ and $ \mathrm{int}(D_2)$, such that $\mathrm{int}(N_1) \cap \mathrm{int}(D_2)\sim N_1 \setminus D_1 \sim S^{n-1}$. We obtain that $M$ has the integral homology of $S^{n}$, and it is simply connected. Therefore, $M$ is a homotopy $n$-sphere by \cite[Theorems 4.32 and 4.33]{Hatcher}. By the topological Poincar\'{e} conjecture, we conclude that $M \simeq S^{n}$. By the PL Poincar\'e conjecture, that holds for $n\neq 4$, we have also $M\simeq_{\rm PL} S^n$; see~\cite[Theorem B]{Sma61}.
\end{proof}

\section{Birational complexity and coregularity}

In this section, we show upper bounds for the birational complexity of a generalized log Calabi--Yau pair.
The main technical statement of this section is the following theorem.

\begin{theorem}\label{thm:mod-with-large-bound}
Let $X$ be a variety of Fano type.
Let $(X,B,\mathbf{M})$ be a generalized log Calabi--Yau pair. There exists a crepant birational map $(Y,B_Y,\mathbf{M}) \dashrightarrow (X,B,\mathbf{M})$, only extracting glc places,
satisfying the following:
\begin{enumerate}
\item $Y$ is $\qq$-factorial,
\item $(Y,B_Y,\mathbf{M})$ is gdlt, and 
\item the number of components of $\lfloor B_Y\rfloor$
is at least
$\rho(Y)-{\rm coreg}(Y,B_Y,\mathbf{M})$.
\end{enumerate} 
\end{theorem} 

\begin{proof}
We proceed by induction
on the dimension of $X$.
The statement is clear in dimension one, where $\rho(X)=1$ and $X$ is a rational curve.
If $(X,B,\mathbf{M})$ is gklt,
then the statement is a consequence of Lemma~\ref{lem:tower-mfs-picard}.
Thus, we may assume that $(X,B,\mathbf{M})$ is glc
but it is not gklt.
Let $(X_0,B_0,\mathbf{M})$ be a gdlt modification of $(X,B,\mathbf{M})$; see Lemma~\ref{lem:gdlt}.
Let $E$ be a prime component of $\lfloor B_0\rfloor$.
We run a $(K_{X_0}+B_0-E+\mathbf{M})$-MMP. 
This MMP terminates with the glc pair $(X_1,B_1,\mathbf{M})$, endowed with the Mori fiber space $\pi_1\colon X_1\rightarrow Z_1$.

First, assume that $Z_1$ is a point.
Then, we have that $\rho(X_1)=1$.
If we ran a $(-E)$-MMP,
by the negativity lemma,
the divisor $E$ is not contracted
by the birational map
$X_0\dashrightarrow X_1$, so
 $\lfloor B_0\rfloor$ has at least one component.
Let $(Y,B_Y,\mathbf{M})$ be a $\qq$-factorial gdlt modification of $(X_1,B_1,\mathbf{M})$.
By Lemma~\ref{lem:bcomp-dlt-mod}, we conclude that
$|\lfloor B_Y\rfloor| - \rho(Y)\geq 0$.
This finishes the proof in this case.

Now, assume that $Z_1$ is a positive dimensional variety.
By Lemmas~\ref{lem:image-FT} and~\ref{lem:contr-FT}, we know that $Z_1$ is a variety of Fano type.
Let $(Z_1,B_{Z_1},\mathbf{N})$ be the generalized pair induced
by the generalized canonical bundle formula.
By induction on dimension, there exists a crepant birational map $(Z_2,B_{Z_2},\mathbf{N}) \dashrightarrow (Z_1,B_{Z_1},\mathbf{N})$, only extracting glc places, such that $(Z_2,B_{Z_2},\mathbf{N})$ is $\qq$-factorial gdlt, and 
\begin{equation}\label{eq:increase|B_Z2|}
|\lfloor B_{Z_2}\rfloor|  \geq \rho(Z_2)-{\rm coreg}(Z_2,B_{Z_2},\mathbf{N}).
\end{equation}
We have the following diagram:
\[
\xymatrix{
X & X_0 \ar[l] \ar@{-->}[r] & X_1\ar[d]_-{\pi_0} &  \\ 
 & & Z_1 & Z_2, \ar@{-->}[l] \\
}
\]
where the horizontal arrows are birational maps
and the vertical arrow is a fibration. By Lemma~\ref{cor:mod-base-mod-tot-space}, there exists a commutative diagram of $\qq$-factorial varieties 
\[
\xymatrix{
X & X_0 \ar[l] \ar@{-->}[r] & (X_1,B_1,\mathbf{M}) \ar[d]_-{\pi_0} &   (X_2,B_2,\mathbf{M})\ar[d]^-{\pi_2} \ar@{-->}[l]_-{\phi}\\ 
 & & (Z_1,B_{Z_1},\mathbf{N}) & (Z_2,B_{Z_2},\mathbf{N}) \ar@{-->}[l]\\
}
\]
satisfying the following conditions
\begin{enumerate}
\item $\phi \colon X_0 \dashrightarrow X_2$ only extracts glc places, 
\item $\pi_1 \colon X_2 \to Z_2$ is a Mori fiber space, and
\item ${\pi_2}^{-1}(\lfloor B_{Z_2}\rfloor)\subseteq \lfloor B_2\rfloor$.
\end{enumerate} 
 By (2), for every component $S$ of $\lfloor B_{Z_2}\rfloor$ there exists a unique component of $\lfloor B_2\rfloor$ that maps onto $S$. Note also that the strict transform of $E$ on $X_1$ is a divisor that is horizontal over $Z_1$.
 By the commutativity of the diagram, we conclude
 that the strict transform of the $\pi_0$-ample divisor $E$ on $X_2$ is a horizontal divisor over $Z_2$. We conclude that 
 \begin{equation}\label{eq:comparing-sum}
 |\lfloor B_2\rfloor| \geq |\lfloor B_{Z_2}\rfloor|+1.
 \end{equation}

By~\cite[Proposition 3.28]{Mor22}, we know that
\begin{equation}\label{eq:comparing-coreg}
{\rm coreg}(X_2,B_2,\mathbf{M})
\geq 
{\rm coreg}(Z_2,B_{Z_2},\mathbf{N}).
\end{equation}
Since $\pi_2$ is a Mori fiber space, we have that $\rho(X_2)=\rho(Z_2)+1$.
Putting this equality
together with inequalities~\eqref{eq:increase|B_Z2|}, ~\eqref{eq:comparing-sum} and 
~\eqref{eq:comparing-coreg},
we conclude that
\begin{equation}\label{eq:boundary-x2}
|\lfloor B_2\rfloor| 
\geq 
\rho(X_2)-{\rm coreg}(X_2,B_2,\mathbf{M}).
\end{equation}
Let $(Y,B_Y,\mathbf{M})$ be a $\qq$-factorial dlt modification of $(X_2,B_2,\mathbf{M})$. Since $Y$ and $X_2$ are both $\qq$-factorial, we obtain the required inequality
\[|\lfloor B_Y\rfloor|-\rho(Y)=|\lfloor B_2\rfloor|-\rho(X_2) \geq -{\rm coreg}(X_2,B_2,\mathbf{M})=-{\rm coreg}(X_0,B_0,\mathbf{M}),\]
by Lemma~\ref{lem:bcomp-dlt-mod}, inequality~\eqref{eq:boundary-x2} and Lemma~\ref{lem:coreg-bir} respectively.
This finishes the proof.
\end{proof}

Now, we turn to prove the main result of this section, which implies Theorem~\ref{theorem:bcomp-bound}.

\begin{theorem}\label{thm:gbcomp-bound}
Let $X$ be a variety of Fano type of dimension $n$.
Let $(X,B,\mathbf{M})$ be a generalized log Calabi--Yau pair.
Then, the following
inequalities hold:
\begin{equation}
\label{eq:1-gen}
0\leq c_{\rm bir}(X,B,\mathbf{M})
\leq {\rm coreg}(X,B,\mathbf{M})+n
\end{equation} 
and 
\begin{equation}
\label{eq:2-gen}
c_{\rm bir}(X,B,\mathbf{M}) \leq 2n,
\end{equation} 
and the inequality is strict if $(X,B)$ is a log Calabi--Yau pair.
\end{theorem}

\begin{proof}
The lower bound in~\eqref{eq:1-gen} follows from Proposition~\ref{cor:bcomp-noneg}. The upper bound is a corollary of Theorem~\ref{thm:mod-with-large-bound}. Indeed, in the notation of Theorem~\ref{thm:mod-with-large-bound}, we have
\[c_{\rm bir}(X,B,\mathbf{M}) \leq c(Y,B_Y,\mathbf{M})=\dim Y+\rho(Y)-|B_Y|\leq \dim Y + {\rm coreg}(Y,B_Y,\mathbf{M})
=
\dim X + {\rm coreg}(X,B,\mathbf{M}).\]
By definition, ${\rm coreg}(X,B,\mathbf{M})\leq n$, and this implies inequality~\eqref{eq:2-gen}. 

The inequality is strict if $(X,B)$ is a log Calabi--Yau pair of Fano type. To this end, observe that the equality
${\rm coreg}(X,B)=n$
holds
if and only if
the pair $(X,B)$ is klt.
By Lemma~\ref{lem:tower-mfs-picard},
there exists a birational contraction
$X\dashrightarrow X'$ 
such that $X'$ is $\qq$-factorial
and $\rho(X')\leq n$.
Since $X\dashrightarrow X'$ is a contraction from a variety of Fano type, then
$X$ is a variety of Fano type.
Let $(X',B')$ be the log Calabi--Yau pair induced on $X'$.
Note that $B'$ is non-trivial as
$B$ is a big effective divisor.
Therefore, we have that 
\[
c(X',B')=\dim X'+\rho(X')-|B'| 
\leq 2n - |B'|<2n,
\]
where the last inequality follows from the fact that $B'$ is a non-trivial divisor.
Thus, we have that 
\[
c_{\rm bir}(X,B) \leq c(X',B')<2n.
\]
This finishes the proof of the theorem.
\end{proof}

\begin{proof}[Proof of Corollary~\ref{introcor:bcomp-coreg-0}]
Let $X$ be a variety of Fano type and absolute coregularity zero.
By~\cite[Theorem 4]{FFMP22}, there exists a $2$-complement $B$ of $X$.
Let $(X,B)$ be the associated log Calabi--Yau pair.
Let $Y\rightarrow X$ be the index one cover of $K_X+B$, i.e., the double cover of $X$ such that
$\pi^*(K_X+B)=K_Y+B_Y \sim 0$.
The pair $(Y,B_Y)$ has coregularity zero, so by~\cite[Theorem 4.2]{FS20}, it is crepant equivalent to $(Z,B_Z)$ where
$Z$ is of Fano type.
The index of a log Calabi--Yau pair is independent of its birational model, so
$K_Z+B_Z\sim 0$.
Then, by Theorem~\ref{theorem:bcomp-bound},
we have that $c_{\rm bir}(X,B)\in \{0,\dots,n\}$.
\end{proof}

\begin{proof}[Proof of Theorem~\ref{theorem:bcomp-zero}]
Let $(X,B)$ be a log Calabi--Yau pair 
with $K_X+B\sim 0$.
Assume that $c_{\rm bir}(X,B)=0$.
By Lemma~\ref{lem:computing-bcomp}, there exists a $\qq$-factorial dlt crepant model $(X',B') \dashrightarrow (X,B)$ with the property that $c(X, B)=0$. By \cite[Theorem 1.2]{BMSZ18}, $X'$ is a toric variety and $B'$ is the full toric boundary; for this we use that $K_{X'} + B' \sim 0$, and not just that $K_X+B$ is $\qq$-linearly trivial. In particular, there is a crepant birational map \[(\pp^n,H_1+\dots+H_{n+1})\dashrightarrow (X',B')\dashrightarrow (X,B).\] 

On the other hand, given the crepant birational map
$(\pp^n,H_1+\dots+H_{n+1})\dashrightarrow (X,B)$, then 
$0 \leq c_{\rm bir}(X,B)\leq c(\pp^n,H_1+\dots+H_{n+1})=0$, i.e., $c_{\rm bir}(X,B)=0$.

Finally, all crepant birational invariants of $(X, B)$ coincide with those of $(\pp^n,H_1+\dots+H_{n+1})$, e.g.,
\begin{align*}
{\rm coreg}(X,B)  = & {\rm coreg}(\pp^n,H_1+\dots+H_{n+1})=0\\
\mathcal{D}(X,B)   \simeq_{\rm PL} & \mathcal{D}(\pp^n,H_1+\dots+H_{n+1})\simeq_{\rm PL}S^{n-1}.
\end{align*}
\end{proof} 

\section{Ample divisors in the boundary}

In this section, we study the existence of positive divisors 
on the reduced boundary $\lfloor B\rfloor$ of log Calabi--Yau pairs $(X,B)$. The first theorem in this section concerns ample divisors
on the reduced boundary, which generalize \cite[Theorem 4.2]{FS20} and \cite[Theorem 4.9]{KX16}.

\begin{theorem}\label{thm:ample-div-bound-with-FT}
Let $(X,B,\mathbf{M})$ be a generalized log Calabi--Yau pair.
Let $X\rightarrow \schbase$ be a fibration 
such that $X$ is of Fano type over $\schbase$
and every glcc
of $(X,B,\mathbf{M})$ dominates $\schbase$.
    Then, there exists a crepant birational map $(X',B',\mathbf{M}) \dashrightarrow (X,B,\mathbf{M})$, only extracting glc places, and a fibration $X' \to W$ over $\schbase$ satisfying the following:
\begin{enumerate} 
\item $X'$ is $\qq$-factorial,
\item $(X',B',\mathbf{M})$ is gdlt,
\item every glcc of $(X',B',\mathbf{M})$ dominates $W$, and 
\item $\lfloor B'\rfloor$ fully supports an ample divisor over $W$.
\end{enumerate} 
\end{theorem} 

\begin{proof}
We prove the statement by induction on the dimension of $X$.
Let $(X_0,B_0,\mathbf{M})$ be a $\qq$-factorial gdlt modification of $(X,B,\mathbf{M})$.
Let $S$ be a component of $B_0$.
We run a $(K_{X_0}+B_0-S+\mathbf{M}_{X_0})$-MMP over $\schbase$.
Since this divisor is not pseudo-effective, the MMP must terminate with a Mori fiber space.
Let $X_0\dashrightarrow X_1 \dashrightarrow X_2 \dashrightarrow \dots \dashrightarrow X_k$ be the steps of this MMP
and $\pi \colon X_k\rightarrow Z$ be the induced Mori fiber space.
Let $(X_k,B_k,\mathbf{M})$ be the associated generalized log Calabi--Yau pair on $X_k$.
Let $S_k$ be the push-forward of $S$ on $X_k$.

First, assume that $Z$ is a point.
Then $X_k$ is a Fano variety of Picard rank one.
By Lemma~\ref{lem:-S-MMP-non-dlt},
we know that the support of $S_k$ 
contains all the non-gdlt centers of $(X_k,B_k,\mathbf{M})$.
By Lemma~\ref{lem:dlt-mod-supp-amp}, 
we know that there exists a $\qq$-factorial 
dlt modification $(X',B',\mathbf{M})$
of $(X_k,B_k,\mathbf{M})$ such that $B'$ supports an ample divisor. This finishes the proof in the case that $Z$ is a point.

Suppose then that $\dim Z > 0$. By Lemma~\ref{lem:image-FT},
we know that $Z$ is of Fano type over $\schbase$.
By the generalized canonical bundle formula
there exists a generalized pair structure
$(Z,B_{Z},\mathbf{N})$ that is generalized log Calabi--Yau.
Note that every
glcc
of $(Z,B_Z,\mathbf{N})$ is horizontal over $\schbase$.
By induction on the dimension, there exists a crepant birational map $(Z',B_{Z'},\mathbf{M}) \dashrightarrow (Z,B_{Z},\mathbf{M})$, only extracting glc places, 
and a fibration $Z'\rightarrow W$
satisfying the following conditions:
\begin{enumerate}
\item $Z'$ is $\qq$-factorial,
\item $(Z',B_{Z'},\mathbf{N})$ is gdlt, 
\item every glcc of $(Z',B_{Z'},\mathbf{N})$ dominates $W$, and 
\item  $\lfloor B_{Z'}\rfloor$ fully supports an ample divisor over $W$.
\end{enumerate} 
We obtain the following commutative diagram
\[
\xymatrix{
(X,B,\mathbf{M})\ar[dd]\ar@{-->}[r] & (X_k,B_k,\mathbf{M})
\ar[d]^-{\pi} &  \\ 
 & (Z,B_Z,\mathbf{N})\ar[dl] & (Z',B_{Z'},\mathbf{N})\ar[d] 
 \ar@{-->}[l]  \\ 
 \schbase & &  W.\ar[ll] \\
}
\]
where the horizontal arrows are birational
and the vertical arrows
are fibrations. Since $X$ is of Fano type over $\schbase$,
we can apply Lemma~\ref{cor:mod-base-mod-tot-space}.
Hence, there exists 
a $\qq$-factorial glc pair
$(Y,B_{Y},\mathbf{M})$
and a crepant birational map
$\phi \colon (Y,B_{Y},\mathbf{M}) \dashrightarrow (X_k,B_k,\mathbf{M})$ satisfying the following conditions:
\begin{enumerate}
\item[(i)] $\pi_Y\colon Y\rightarrow Z'$ is a Mori fiber space, 
\item[(ii)] $\pi_Y^{-1}\lfloor B_{Z'}\rfloor\subseteq \lfloor B_{Y}\rfloor$, and
\item[(iii)] every glcc of $(Y,B_{Y},\mathbf{M})$ dominates $W$.
\end{enumerate}
This leads to the following commutative diagram
\[
\xymatrix{
(X,B,\mathbf{M})\ar[dd]\ar@{-->}[r] & (X_k,B_k,\mathbf{M})
\ar[d]^-{\pi} &  \ar@{-->}[l]_-{\phi} (Y,B_{Y},\mathbf{M}) \ar[d]^-{\pi_Y}\\ 
 & (Z,B_Z,\mathbf{N})\ar[dl] & (Z',B_{Z'},\mathbf{N})\ar[d] 
 \ar@{-->}[l]  \\ 
 \schbase & &  W.\ar[ll] \\
}
\]

Let $S_{Y}$ be the strict transform of $S_k$ on $Y$. By construction,
$S_{Y}$ dominates $Z'$.
Since
$\rho(Y/Z')=1$, 
we conclude that $S_{Y}$ is ample over $Z'$.
Furthermore, since $\phi$ is an isomorphism at the generic point of $Z$,
the support of $S_{Y}$
contains every non-gdlt
center of $(Y,B_{Y},\mathbf{M})$
that dominates $Z'$.
On the other hand,
by (ii) above, $\lfloor B_{Y}\rfloor$
contains every vertical
glcc of $(Y,B_{Y},\mathbf{M})$.
Hence, the divisor
$\lfloor B_{Y}\rfloor$
supports an ample divisor over $W$
and contains every non-gdlt center of $(Y,B_{Y},\mathbf{M})$.
By Lemma~\ref{lem:dlt-mod-supp-amp},
there exists a $\qq$-factorial gdlt modification
$(X',B',\mathbf{M})$ of $(Y,B_{Y},\mathbf{M})$
for which $\lfloor B'\rfloor$ supports an ample over $W$.
By construction, every glcc of $(X',B',\mathbf{M})$ dominates $W$.
This finishes the proof.
\end{proof} 

Using previous work due to Filipazzi and Svaldi, we can get rid of the Fano type assumption on the previous theorem.
The following theorem is a version of Theorem~\ref{theorem:supp-amp}
in the context of generalized pairs.

\begin{theorem}\label{thm:ample-div-bound}
Let $(X,B,\mathbf{M})$ be a generalized log Calabi--Yau pair.
There exists a crepant birational map $(X',B',\mathbf{M}) \dashrightarrow (X,B,\mathbf{M})$ and a fibration
$X'\rightarrow W$ satisfying the following:
\begin{enumerate}
\item $X'$ is $\qq$-factorial,
\item $(X',B',\mathbf{M})$ is gdlt,
\item every glcc of $(X',B',\mathbf{M})$ dominates $W$, and 
\item $\lfloor B'\rfloor$ fully supports an ample divisor over $W$.
\end{enumerate}   
\end{theorem} 

\begin{proof}
By~\cite[Theorem 4.2]{FS20},
there exists a crepant birational model
$(X_1,B_1,\mathbf{M})$ of $(X,B,\mathbf{M})$
and a fibration $X_1\rightarrow \schbase$ such that
\begin{itemize}
\item 
$(X_1,B_1,\mathbf{M})$ is gdlt, 
\item every glcc of $(X_1,B_1,\mathbf{M})$ dominates $\schbase$, and 
\item $X_1$ is of Fano type over $\schbase$.
\end{itemize} 
Then, the statement follows from Theorem~\ref{thm:ample-div-bound-with-FT}.
\end{proof} 

In the next theorem, we show that 
Theorem~\ref{thm:ample-div-bound}.(4)
can be improved
if we have some control
on the birational complexity
of $(X,B,\mathbf{M})$.

\begin{theorem}\label{thm:proper-supp-amp-with-FT}
Let $(X,B,\mathbf{M})$ be a $\qq$-factorial generalized log Calabi--Yau pair.
Let $X\rightarrow \schbase$ be a fibration of Fano type.
Assume that every glcc of $(X,B,\mathbf{M})$ dominates $\schbase$
and that $|\lfloor B\rfloor|-\rho(X/\schbase)>0$.
Then, there exists a crepant birational {map $\phi \colon (X',B',\mathbf{M}) \dashrightarrow (X,B,\mathbf{M})$ over $\schbase$, only extracting glc places,} 
and a fibration $X'\rightarrow W$
over $\schbase$
satisfying the following:
\begin{enumerate}
\item  $X'$ is $\qq$-factorial,
\item 
$(X',B',\mathbf{M})$ is gdlt, 
\item every glcc of $(X',B',\mathbf{M})$
dominates $W$, and 
\item $\lfloor B'\rfloor$ properly supports an effective ample divisor over $W$.
\end{enumerate}
\end{theorem} 

\begin{proof}
We proceed by induction on the dimension of $X$. Let $(X_0,B_0,\mathbf{M})$ be a $\qq$-factorial gdlt crepant birational model of $(X,B,\mathbf{M})$ such that $X_{0} \dashrightarrow X$ only extracts glc places.
Let $X_0\rightarrow Z_0$ be a fibration over $\schbase$, and $S$ a component of $\lfloor B_0\rfloor$ dominating $Z_0$.
Let $X_0\dashrightarrow X_1$ be the outcome of a $(K_{X_0}+B_0-S+\mathbf{M}_{X_0})$-MMP over $Z_0$,
and $X_1\rightarrow Z_1$ be the Mori fiber space. {We obtain a commutative diagram as follows:}
\[
\xymatrix{
X & \ar@{-->}[l] X_0\ar@{-->}[r]\ar[d] & X_1 \ar[d] \\ 
& Z_0\ar[d] & Z_1\ar[l]\ar[ld] \\ 
& \schbase.
}
\]
Assume that the relative dimension of $X_1\rightarrow Z_1$ is minimal among the fibrations over $\schbase$ obtained in the previous way.
Let $(X_1,B_1,\mathbf{M})$ be the generalized pair induced on $X_1$.
Let $(Z_1,B_{Z_1},\mathbf{N})$ be the generalized log Calabi--Yau pair induced on $Z_1$ by the generalized canonical bundle formula.
\color{black}
By Lemma~\ref{lem:image-FT} {and ~\ref{lem:contr-FT},} we know that $Z_1$ is of Fano type over $\schbase$. 

By Lemma~\ref{lem:bcomp-dlt-mod},
we conclude that $|\lfloor B_0\rfloor|-\rho(X_0/\schbase)>0$.
In particular, we have that $|\lfloor B_1\rfloor|-\rho(X_1/\schbase)>0$.
There are two cases,
depending on the number of components of $\lfloor B_1\rfloor$ that are horizontal over $Z_1$.\\

\textit{Case 1:} There is a unique component of $\lfloor B_1\rfloor$ that is horizontal over $Z_1$.\\ 

Note that this component must be 
$S_1$ the image of $S$ on $X_1$.
We conclude that $\lfloor B_1\rfloor$ has at least 
$\rho(Z_1/\schbase)+1$ vertical components over $Z_1$.
Thus, we have that $|\lfloor B_{Z_1}\rfloor|-\rho(Z_1/\schbase)>0$.
Hence, the generalized pair $(Z_1,B_{Z_1},\mathbf{N})$ satisfies all the hypothesis of the statement.
By induction on the dimension, there exist a $\qq$-factorial gdlt crepant birational map $(Z_2,B_{Z_2},\mathbf{N}) \dashrightarrow (Z_1,B_{Z_1},\mathbf{N})$, only extracting glc places,
and a fibration $Z_2\rightarrow W$ over $\schbase$
with the following conditions:
\begin{itemize}
\item every glcc of $(Z_2,B_{Z_2},\mathbf{N})$ dominates $W$, and 
\item the divisor $\lfloor B_{Z_2}\rfloor$ properly supports an effective ample divisor over $W$.
\end{itemize} 
Note that $X_1$ is of Fano type over $\schbase$, 
so we may apply Lemma~\ref{cor:mod-base-mod-tot-space}.
By doing so, we obtain a crepant birational map $(X_2,B_2,\mathbf{M})\dashrightarrow (X_1,B_1,\mathbf{M})$, only extracting glc places, and  a Mori fiber space
$\pi_2\colon X_2\rightarrow Z_2$ such that
$\pi_2^{-1}\lfloor B_{Z_2}\rfloor\subseteq \lfloor B_2\rfloor$.

Let $(X_3,B_3,\mathbf{M})$ be 
a $\qq$-factorial gdlt modification of $(X_2,B_2,\mathbf{M})$.
Recall that there is an effective divisor $A_{Z_2}$ that is properly
supported on $\lfloor B_{Z_2}\rfloor$ that is ample over $W$.
Furthermore, there is a component $S_3$
of $B_3$ that is horizontal over $Z_3$.
We run a $(K_{X_3}+B_3-S_3+\mathbf{M}_{X_3})$-MMP over $Z_2$.
Let $X_3\dashrightarrow X_4 \dashrightarrow \dots \dashrightarrow X_k$ be the steps of such MMP
and $X_k\rightarrow Z_k$ be the associated Mori fiber space.
We obtain a commutative diagram as follows:
\[
\xymatrix{
X & \ar@{-->}[l] X_3\ar@{-->}[r]\ar[d] & X_k \ar[d]^-{\pi_k} \\ 
& Z_2\ar[d] & Z_k\ar[l]\ar[ld] \\ 
& W,
}
\]
where $Z_k\rightarrow Z_2$ is a projective birational morphism, 
by the minimality assumption on 
the relative dimension of $X_1\rightarrow Z_1$.
Let $(Z_k,B_{Z_k},\mathbf{N})$ be the generalized pair induced on 
$Z_k$. By construction,
the morphism $(Z_k,B_{Z_k},\mathbf{N})\rightarrow (Z_2,B_{Z_2},\mathbf{N})$ is crepant. Since $X_3 \to Z_2$ is a composition of a gdlt modification and a Mori fiber space, every prime divisor on $X_3$, vertical over $Z_2$, is either the preimage of a prime divisor on $Z_2$, or contained in $\lfloor B_2\rfloor$. Therefore,
every prime divisor on $Z_k$, exceptional over $Z_2$, is contained in $\lfloor B_{Z_k}\rfloor$.
By Lemma~\ref{lem:prop-supp-ampl-mod},
we conclude that $\lfloor B_{Z_k}\rfloor$ properly supports an ample divisor $A_{Z_{k}}$ over $W$.
By Lemma~\ref{lem:-S-MMP-non-dlt}, the non-gdlt locus of $(X_k,B_{X_k},\mathbf{M})$ is contained in $S_k$, the strict transform of $S_2$ on $X_k$.
Note that the divisor $E:=\pi_k^* A_{Z_k}+\epsilon S_k$
 satisfies the following conditions:
\begin{itemize}
\item $E$ is effective,
\item $E$ is properly supported on $\lfloor B_k\rfloor$, 
\item $E$ contains every non-gdlt center of $(X_k,B_{k},\mathbf{M})$, and
\item $E$ is ample over $W$.
\end{itemize}
By Lemma~\ref{lem:dlt-mod-supp-amp}, there is a $\qq$-factorial gdlt modification $(X',B',\mathbf{M})$ of $(X_k,B_{k},\mathbf{M})$ such that
$\lfloor B'\rfloor$ properly supports an effective divisor that is ample over $W$.
Note that every glcc of $(X',B',\mathbf{M})$ dominates $W$.
This finishes the proof in this case.\\

\textit{Case 2:} There are at least two components of $\lfloor B_1\rfloor$ that are horizontal over $Z_1$.\\ 
 
Since $Z_1\rightarrow \schbase$ is a fibration of Fano type, we can apply Theorem~\ref{thm:ample-div-bound-with-FT}. Hence, there exist a crepant birational map $\phi \colon (Z_2,B_{Z_2},\mathbf{N})\dashrightarrow (Z_1,B_{Z_1},\mathbf{N})$, and a  fibration
$Z_2\rightarrow W$, satisfying the following conditions:
\begin{enumerate}
\item $\phi$ only extracts glc places, 
\item $(Z_2,B_{Z_2},\mathbf{N})$ is a $\qq$-factorial gdlt pair,
\item every glcc of $(Z_2,B_{Z_2},\mathbf{N})$ dominates $W$, and 
\item $\lfloor B_{Z_2}\rfloor$ fully supports an ample divisor over $W$.
\end{enumerate}
By Lemma~\ref{cor:mod-base-mod-tot-space}, we obtain a $\qq$-factorial
glc pair $(X_2,B_2,\mathbf{M})$ 
that is crepant birational
to $(X_1,B_1,\mathbf{M})$ and satisfies the following conditions:
\begin{enumerate}
\item[(i)] $X_2\dashrightarrow X_1$ only extracts glc places, 
\item[(ii)] 
$\pi_3\colon X_2\to Z_2$ is a Mori fiber space, 
\item[(iii)] 
$\pi_3^{-1}\lfloor B_{Z_2}\rfloor \subseteq \lfloor B_2\rfloor$, and 
\item[(iv)] there is a component $S_2$ of $B_2$ that contains all horizontal non-gdlt centers of $(X_2,B_2,\mathbf{M})$.
\end{enumerate}
Indeed, by construction, the MMP involved in Lemma~\ref{cor:mod-base-mod-tot-space} induces an isomorphism between the general fibers of $X_2 \to Z_2$ and $X_1 \to Z_1$, where $S_1$ contains all horizontal non-gdlt centers of $(X_1,B_1,\mathbf{M})$ by Lemma~\ref{lem:-S-MMP-non-dlt}.

By (ii) and (iv), the divisor 
$E_2:=\pi_3^*B_{Z_2}+S_2$
supports an ample divisor.
By (iii) and the fact that $\lfloor B_2\rfloor$
contains at least two horizontal divisors over $Z_2$, 
we conclude that $\lfloor B_2\rfloor$ properly supports
an effective divisor that is ample over $W$.
By $(3)$, every glcc of $(X_2,B_2,\mathbf{M})$ dominates $W$.
By (iv), the divisor $E_2$ contains all non-gdlt centers of $(X_2,B_2,\mathbf{M})$.
Thus, we may apply Lemma~\ref{lem:dlt-mod-supp-amp}, to conclude that there exists a $\qq$-factorial
gdlt modification $(X',B',\mathbf{M})$ of $(X_2,B_2,\mathbf{M})$
for which $\lfloor B'\rfloor$ properly supports
and effective divisor that is ample over $W$.
This finishes the proof.
\end{proof} 

We turn to prove a more general statement without the Fano type assumption. 

\begin{theorem}\label{thm:proper-supp-amp}
Let $(X,B,\mathbf{M})$ be a $\qq$-factorial generalized log Calabi--Yau pair with $X$.
Assume that $|\lfloor B\rfloor|-\rho(X)>0$.
Then, there exists a crepant birational map $\phi \colon (X',B',\mathbf{M}) \dashrightarrow (X,B,\mathbf{M})$, only extracting glc places, 
and a fibration
$X'\rightarrow Z$ satisfying the following:
\begin{enumerate}
\item $X'$ is $\qq$-factorial,
\item $(X',B',\mathbf{M})$ is gdlt,
\item every glcc of $(X',B',\mathbf{M})$ dominates $Z$, and 
\item $\lfloor B'\rfloor$ properly supports an effective divisor
which is ample over $Z$.
\end{enumerate} 
\end{theorem} 

\begin{proof}
By~\cite[Theorem 4.2]{FS20},
there exists a $\qq$-factorial crepant
birational model
$(X_1,B_1,\mathbf{M})$ of $(X,B,\mathbf{M})$
and a fibration
$X_1\rightarrow \schbase$ such that
\begin{enumerate}
\item[(i)]
$(X_1,B_1,\mathbf{M})$ is gdlt,
\item[(ii)] $X_1\dashrightarrow X$ only extracts glc places of $(X,B,\mathbf{M})$, 
\item[(iii)] every glcc of $(X_1,B_1,\mathbf{M})$ dominates $\schbase$, and 
\item[(iv)] $X_1$ is of Fano type over $\schbase$.
\end{enumerate} 
By (ii) and Lemma~\ref{lem:bcomp-dlt-mod}, we have that 
$|\lfloor B_1\rfloor|-\rho(X_1)>0$.
In particular, we have that
$|\lfloor B_1\rfloor|-\rho(X_1/\schbase)>0$.
Then we apply Theorem~\ref{thm:proper-supp-amp-with-FT} to conclude the proof.
\end{proof} 

Recall that a dual complex of a dlt pair is a regular
$\Delta$-complex. It is simplicial if and only if every face is prescribed by the collection
of its vertices, i.e., if and only if the intersection of any collection of divisors is either empty or irreducible (equivalently connected). Here we observe that the dual complex $\mathcal{D}(X',B',\mathbf{M})$ can be made simplicial preserving the properties (1)-(4) in Theorem~\ref{thm:proper-supp-amp}.

\begin{lemma}\label{lem:simplicial} There exists generalized pairs $(X',B',\mathbf{M})$ as in Theorem \ref{thm:proper-supp-amp} with the additional property that the dual complex $\mathcal{D}(X',B',\mathbf{M})$ is simplicial. More precisely, one can replace $(X',B',\mathbf{M})$ with a pair $(X'',B'',\mathbf{M})$ satisfying all properties (1)-(4) in Theorem~\ref{thm:proper-supp-amp} whose dual complex  $\mathcal{D}(X'',B'',\mathbf{M})$ is a baricentric subdivision of $\mathcal{D}(X',B',\mathbf{M})$.
\end{lemma}
\begin{proof}
We follow closely \cite[Remark 10]{dFKX17} and \cite[Lemma 4.3.3]{MMS2022}. Let $(X',B',\mathbf{M})$ be a generalized pair as in Theorem~\ref{thm:proper-supp-amp}. Let $Z_i \subset X'$ be the union of the $i$-dimensional strata of $\lfloor B' \rfloor$. Consider the sequence of blowups
\[\Pi \colon X'_{n-1} \xrightarrow{\pi_{n-2}} X'_{n-2} \to \ldots \xrightarrow{\pi_0} X'_{0} \coloneqq X',\]
where $\pi_i \colon X'_{i+1} \to X'_{i}$ denotes the blow-up of the birational transform $(\pi_{0} \circ \ldots \pi_{i-1})^{-1}_*(Z_{i})$. Let $p \colon Y \to X'_{n-1}$ be a log resolution of the pair $(X'_{n-1}, \Pi^{-1}_* B' + \Exc(\Pi))$ that is an isomorphism over the snc locus of $(X_{n-1}, \Pi^{-1}_* B' + \Exc(\Pi))$. Let $E \coloneqq \sum^{n-2}_{i=0} E_i + \sum_{j \in J} G_j$ be the exceptional locus of $\Pi \circ p$, where $E_i$ are the strict transform of the reduced sum of the $\pi_i$-exceptional divisors whose center is contained in the snc locus of $(X', B', \mathrm{M}_{X'})$. By construction, $(\Pi\circ p)(G_j)$ does not lie in the snc locus of $(X', B', \mathrm{M}_{X'})$, so by definition of dlt pair, the log discrepancy $a_{G_j}(X', B', \mathbf{M})$ is positive. Note that
\[K_{Y} + E + \mathbf{M}_Y\sim_{\Pi \circ p, \qq} \sum_{j \in J} a_{G_j}(X', B', \mathbf{M}) G_j.\]
The $(\Pi\circ p)$-relative $(K_{Y} + E + \mathbf{M}_Y)$-MMP terminates with a $\qq$-factorial gdlt modification $g \colon (X'', B'', \mathbf{M}) \to (X', B', \mathbf{M})$, which contracts all and only the divisors $G_j$; use for instance \cite[Lemma 4.4]{BZ16} and \cite[\S 1.35]{Kol13}. Observe that the image of the generic point of the strata of $B''$ is contained in the snc locus of $(X', B', \mathrm{M}_{X'})$. Hence, $\mathcal{D}(X'',B'',\mathbf{M})$ is a barycentric subdivision of $\mathcal{D}(X',B',\mathbf{M})$ from \cite[Remark 10]{dFKX17}. 

Note that $(X'',B'',\mathbf{M})$ enjoys property (4) of Theorem~\ref{thm:proper-supp-amp}, as $(X',B',\mathbf{M})$ does too. Indeed, let $A$ be an effective ample divisor properly supported on $\lfloor B'\rfloor$. Let $D$ be a component of $\lfloor B'\rfloor$ not contained in the support of $A$. Since $g \colon X'' \to X'$ is a birational morphism of $\qq$-factorial varieties, then the divisor
\[g^*(A+\epsilon(\lfloor B'\rfloor-D))-\delta E_{X''}\]
on $X''$ is ample, where $E_{X''}$ is the pushforward of $E$ on $X''$, and $0 < \delta \ll \epsilon \ll 1$.
\end{proof}

To conclude this section, we give a proof of Theorem~\ref{theorem:proper-supp-amp}. 

\begin{proof}[Proof of Theorem~\ref{theorem:proper-supp-amp}]
Let $(X,B)$ be a log Calabi--Yau pair of dimension $n$.
Assume that $K_X+B\sim 0$ and $c_{\rm bir}(X,B)<n$.
By Lemma~\ref{lem:computing-bcomp}, there exists a crepant birational model
$(X',B')$ of $(X,B)$
for which $c(X',B')=c_{\rm bir}(X,B)<n$.
Furthermore, we may assume that $(X',B')$ is dlt and $X'$ is $\qq$-factorial.
Since $K_X+B\sim 0$, we have that 
$K_{X'}+B'\sim 0$ so $B'$ is reduced.
From the inequality
\[
c(X',B')<n
\]
we deduce that 
$|B'|>\rho(X')$.
Thus, the statement follows from Theorem~\ref{thm:proper-supp-amp} and Lemma~\ref{lem:simplicial}.
\end{proof} 

\section{Positive divisors in the boundary} 

In this section, we study the existence of positive divisors in the boundary of a log Calabi--Yau pair.
We introduce the following theorem which is a generalization of Theorem~\ref{theorem:big-decomp} to the setting of generalized pairs.

\begin{theorem}\label{theorem:big-decomp-gen}
Let $X$ be an $n$-dimensional klt variety and let $(X,B,\mathbf{M})$ be a generalized log Calabi--Yau pair.
Let $\sum_{i\in I}B_i$ be a decomposition of $\lfloor B\rfloor$ into big divisors.
Then, we have that $|I|\leq n+1$.
Furthermore, if the equality holds, then
the b-nef divisor $\mathbf{M}$ is trivial and 
we have a crepant birational contraction
$(X,B)\dashrightarrow (T,B_T)$ where $T$ is a weighted projective space with its toric boundary $B_{T}$. 
In particular, $\mathcal{D}(X,B)$ is PL-homeomorphic to the boundary of $n$-dimensional simplex.
\end{theorem}

\begin{proof}
Since $X$ has klt singularities,
we may pass to a small $\mathbb{Q}$-factorialization.
We replace the components of the decomposition $B_i$ with their strict transforms.
Thus, we may assume that $X$ is itself $\mathbb{Q}$-factorial.
Assume that the decomposition is non-trivial
and let $B_1$ be a component of the decomposition. 
We run a $(K_X+B-B_1+\mathbf{M}_X)$-MMP.
The divisor $(X,B-B_1,\mathbf{M})$ is glc and $K_X+B-B_1+\mathbf{M}_X$ is not pseudo-effective, hence this minimal model program must terminate with a Mori fiber space.
Let $X\dashrightarrow X_1 \dashrightarrow X_2 \dashrightarrow \dots \dashrightarrow X_k$ be the steps of this MMP and $X_k\rightarrow Z$ be the induced Mori fiber space.
Let $B_k$ be the push-forward of $B$ to $X_k$.
Let $\sum_{i\in I} B_{k,i}$ be the decomposition
induced on $\lfloor B_k\rfloor$.
Note that as each $B_i$ is big then none of these divisors can be contracted by the birational contraction $X\dashrightarrow X_k$.
Thus, each $B_{k,i}$ is a big divisor.
Hence, we have a decomposition of $\lfloor B_k\rfloor$ into $|I|$ big divisors.
Again, as each $B_{k,i}$ is big, then 
all these divisors are horizontal over $Z$.
Let $F$ be the general fiber of $X_k\rightarrow Z$.
Let $\Gamma_F$ be the restriction of $\lfloor B_k\rfloor$ to $F$.
Then, the pair $(F,\Gamma_F)$ is either log Fano or log Calabi--Yau.
Furthermore, the divisor $\Gamma_F$ is reduced with at least $|I|$ components. By the non-negativity of the fine complexity $\overline{c}(F,\Gamma_{F}, \mathbf{M}|_{F})$ (cf Theorem \ref{thm:comp-noneg}),  we conclude that 
\begin{equation}
\label{eq:upper-bound-f}
|I| \leq \dim F + \rho(X_{k}/Z) = \dim F + 1 \leq n+1.
\end{equation}
This proves the first statement of the theorem.
Now, we assume that the equality holds.
Then the sequence of inequalities~\eqref{eq:upper-bound-f} implies that $\dim F= n$, so $Z$ is a point.
Thus $X_k$ has Picard rank one.
By Theorem~\ref{thm:comp-noneg}, we conclude that 
$\mathbf{M}$ is the trivial b-nef divisor
and $(X_k,\lfloor B_k\rfloor)$ is a toric pair.
We argue that $B_k$ must be reduced in this case.
Indeed, by construction, $\lfloor B_k \rfloor$ has $n+1$ components.
If $\{B_k\}$ was non-trivial, then 
the log Calabi--Yau pair $(X,B_k)$ would have negative complexity, contradicting Theorem~\ref{thm:comp-noneg}.
Thus, we conclude that $X_k$ is a projective toric variety of Picard rank one, i.e., a weighted projective space
and $B_k$ is the toric boundary.
By construction, the birational map
$(X,B)\dashrightarrow (X_k,B_k)$ is crepant. In particular, $\mathcal{D}(X,B)$ is PL-homeomorphic to $\mathcal{D}(X_{k},B_{k})$, which is isomorphic to the boundary of $n$-dimensional simplex.
\end{proof}

Now, we turn to give an upper bound for the number of components of movable decompositions of the boundary divisor of a log Calabi--Yau pair.

\begin{theorem}\label{theorem:mov-decomp-gen}
Let $X$ be an $n$-dimensional klt variety and let $(X,B,\mathbf{M})$ be a generalized log Calabi--Yau pair.
Let $\sum_{i\in I} B_i$ be a decomposition of $\lfloor B\rfloor$ into movable divisors.
Then, we have that $|I|\leq 2n$.
Furthermore, if the equality holds,
then the b-nef divisor $\mathbf{M}$ is trivial
and there is an isomorphism
$(X,B)\simeq ((\mathbb{P}^1)^n,B_T)/A$
where $B_T$ is the toric divisor
and $A\leqslant \mathbb{G}_m^n\leqslant {\rm Aut}((\mathbb{P}^1)^n,B_T)$.
In particular, $\mathcal{D}(X,B)$ is PL-homemorphic to an $n$-dimensional orthoplex.
\end{theorem} 

\begin{proof}
We proceed by induction on the dimension, the statement being clear in dimension one.
We may pass to a small $\mathbb{Q}$-factorialization of $X$ and assume that each component $B_i$ of the decomposition is $\qq$-Cartier.
Let $B_1$ be a component of the decomposition.
We run a $(K_X+B-B_1+\mathbf{M}_X)$-MMP.
Since the divisor $K_X+B-B_1+\mathbf{M}_X$ is not pseudo-effective, the minimal model program must terminate with a Mori fiber space. 
Let $X\dashrightarrow X_1 \dashrightarrow X_2 \dashrightarrow \dots \dashrightarrow X_k$ be the steps of this MMP and $\phi\colon X_k\rightarrow Z$ be the induced Mori fiber space.
Let $B_k$ be the push-forward of the divisor $B$ on $X_k$.
By construction, the generalized pair
$(X_k,B_k,\mathbf{M})$ is glc and Calabi--Yau.
Note that no component $B_i$ of the decomposition of $\lfloor B\rfloor$ is contracted by the MMP
as these divisors are movable.
Therefore, the divisor $\lfloor B_k\rfloor$ admits a decomposition into movable divisors with $i:=|I|$ components.
Let $B_{k,1},\dots,B_{k,i}$ be the components of the decomposition of $\lfloor B_k\rfloor$.
Up to re-ordering these components,
we may assume that $B_{k,1},\dots,B_{k,f}$ are horizontal over the base 
while
$B_{k,f+1},\dots,B_{k,i}$ are vertical over the base.
For each $j\in \{f+1,\dots,i\}$, 
we let $B_{Z,j}$ be the unique effective divisor on $Z$ 
for which $\phi^*B_{Z,j}=B_{k,j}$.
By construction, the divisors $B_{Z,j}$ are movable
and $B_{Z,j}\wedge B_{Z,j'}=0$ for every $j\neq j'$.
Let $(Z,B_Z,\mathbf{N})$ be the log Calabi--Yau glc pair induced by the generalized canonical bundle formula. The divisors $B_{Z,j} \subseteq \lfloor B_Z\rfloor$, with $j\in \{f+1,\dots,i\}$,
induce a decomposition
of $\lfloor B_Z\rfloor$ into movable divisors.
By induction on the dimension 
we have that
\begin{equation}\label{ineq:base}
i-f \leq 2\dim Z. 
\end{equation}
On the other hand, 
let $F$ be the general fiber of $\phi$.
Let $\Gamma_F$ be the restriction of $\lfloor B_k\rfloor$ to $F$.
Then, the pair $(F,\Gamma_F)$ is either log Fano or log Calabi--Yau. 
Furthermore, the divisor $\Gamma_F$ has at least $f$ components. 
By the non-negativity of the fine complexity $\overline{c}(F,\Gamma_{F}, \mathbf{M}|_{F})$ (cf Theorem \ref{thm:comp-noneg}),   we conclude that 
\begin{equation}\label{ineq:fiber}
f \leq \dim F+\rho(X_k/Z)= \dim F+1. 
\end{equation}
Adding inequality~\eqref{ineq:base}
and inequality~\eqref{ineq:fiber}, 
we obtain
\begin{equation}\label{ineq:2n}
i = (i-f)+f \leq 2\dim Z + \dim F + 1 
\leq 2n.
\end{equation}
This finishes the proof of the first statement.

From now on, we assume that the equalities $i=|I|=2\dim X$ hold. 
Then, all the inequalities in equation~\eqref{ineq:2n} are indeed equalities.
We conclude that $\dim F=1$.
Thus, the Mori fiber space $X_k\rightarrow Z$ is a conic bundle.
Furthermore, the equality must hold in~\eqref{ineq:base}.
Thus, by induction on the dimension, 
we conclude that 
$\mathbf{N}$ is the trivial b-nef divisor, 
$Z$ is a projective toric variety
whose moment polytope is an $n$-dimensional parallelotope,
and $B_Z$ is the toric boundary.
Moreover, the equality must hold in~\eqref{ineq:fiber}.
This implies that $\lfloor B_k \rfloor$ 
must have $2$ horizontal components over $Z$.
We conclude that 
$\lfloor B_k\rfloor$ has exactly $2n$ prime components. This value is exactly 
\[
\dim X_k + \rho(X_k) = 
\dim Z + \rho(Z)+2 = 2n.
\]
Thus, we may apply Theorem~\ref{thm:comp-noneg}
to conclude that $\mathbf{M}$ is the trivial b-nef divisor 
and $(X_k,B_k)$ is a toric log Calabi--Yau pair.

We turn to argue that the projective toric variety $X_k$ is associated to an $n$-dimensional parallelotope.
Let $\Sigma_0 \subset N_\qq \simeq \qq^{n-1}$ be the fan defining the projective toric variety $Z$, i.e., $X(\sigma_0)\simeq Z$.
Since the moment polytope associated to $\Sigma_0$ is a parallelotope,
we conclude that 
\[
\Sigma_0(1):=\{v_1,\dots,v_{n-1},\dots,-v_1,\dots,-v_{n-1}\},
\]
where the vectors $v_1,\dots,v_{n-1}$ gives a basis of $N_\qq$.
Let $\Sigma \subset N_\qq \oplus \qq$ be the fan defining $X_k$.
Since the conic bundle $X_k\rightarrow Z$ must be induced by a linear maps between fans, 
we have that 
\[
\Sigma(1):=\{e_n,-e_n, v_1+c_1e_n, \dots ,v_n+c_ne_n, -v_1-k_1e_n,\dots, -v_n-k_ne_n\}, 
\]
for certain rational numbers $c_i$ and $k_i$.
It suffices to show that the vectors of 
$\Sigma(1)\setminus \{e_n,-e_n\}$ lie in a common hyperplane of $N_\qq$.
As the toric divisor $D(e_n)$ associated to $e_n\in \Sigma(1)$ is movable, there exists an effective torus invariant divisor $E$ such that 
$E$ does not contain $D(e_n)$ on its support and $D(e_n)\sim_\qq E$.
This means that there exists a linear functional $H_1$ on $N_\qq\oplus \qq$ such that $H_1(e_n)>0$ and $H_1(v)\leq 0$ for every $v\in \Sigma(1) \setminus\{e_n\}$.
Applying the same argument to $-e_n$, we conclude that there exists a linear functional $H_2$ on $N_\qq \oplus \qq$ such that $H_2(-e_n)>0$ and $H_2(v)\leq 0$ for every $v\in \Sigma(1) \setminus\{-e_n\}$.
For each $i \in \{1,\dots,n\}$, 
we have that 
\[
-k_iH_1(e_n)\leq H_1(v_i) \leq c_iH_1(-e_n).  
\]
This implies that $0\leq (c_i-k_i)H_1(-e_n)$ so $c_i-k_i\leq 0$.
Analogously, we have that 
\[
-k_iH_2(e_n) \leq H_2(v_i) \leq c_iH_2(-e_n).
\]
This implies that $0\leq (c_i-k_i)H_2(-e_n)$ so $c_i-k_i\geq 0$. 
We conclude $c_i=k_i$ for each $i$. 
Thus, the rays of $\Sigma(1)\setminus\{e_n,-e_n\}$ are contained in a hyperplane. 
We conclude that $X_k$ is the projective toric variety whose moment polytope is an $n$-dimensional parallelotope.

Finally, we argue that the birational contraction $(X,B)\dashrightarrow (X_k,B_k)$ is indeed an isomorphism. 
The final step of this minimal model program $X_{k-1}\dashrightarrow X_k$
is either a flip or a divisorial contraction.  
If it was a flip, then $X_k$ would admit a big divisor which is not semiample. Since the moment polytope of $X_{k}$ is a parellelotope, the effective cone and nef cone of $X_k$ agree, which excludes the existence of such divisors. Hence, the last step 
$\pi \colon X_{k-1}\to X_k$ of the MMP must be a divisorial contraction, with exceptional divisor $E$. By construction, the strict transform of each irreducible divisor of $B_k$ on $X_{k-1}$ is a movable divisor in $X_{k-1}$.
Indeed, each divisor is the push-forward on $X_{k-1}$ of a movable component on a decomposition of $\lfloor B \rfloor$ on $X$ (in particular none of the components contains $E$ as they are movable).
Since the contraction
$\pi\colon (X_{k-1},B_{k-1}) \to (X_k,B_k)$ is crepant, and $K_{X_{k}}+B_{k}=K_{X_{k}}+\lfloor B_{k} \rfloor \sim 0$, then the log pullback of $(X_k,B_k)$ on $X_{k-1}$ is reduced, i.e., $a_E(X_k,B_k)\in \{0,1\}$. This means that the center 
$c_E(X_k)$ of the divisor $E$ on $X_{k}$ lies in $B_k$; otherwise $E$ would be a terminal valuation, i.e., $a_E(X_k,B_k)>1$. Let $S$ be an irreducible component of $B_{k}$ containing $c_E(X_k)$. Let $\pi_S \colon X_k\rightarrow \mathbb{P}^1$ be the fibration whose fiber over $0 \in \mathbb{P}^1$ is $S$, which is induced by the linear projection of the moment parallelotope of $X_{k}$ onto one of its edges.
Consider the composition 
$p:=\pi_S\circ \pi \colon X_{k-1}\rightarrow \mathbb{P}^1$.
The strict transform $S_{k-1}$ of $S$ in $X_{k-1}$ is one of the two components of the fiber of $p$ over $0$. 
This implies that $S_{k-1}$ is covered by $(S_{k-1})$-negative curves, so $S_{k-1}$ is not a movable divisor, which is a contradiction.
We conclude that the last step
of the MMP $X\dashrightarrow X_k$ is neither a flip nor a divisorial contraction.
Therefore the MMP is indeed trivial, i.e., there is an isomorphism
$(X,B)\simeq (X_k,B_k)$.
\end{proof} 

\section{Dimension of the dual complex} 

In this section, we study the relation between the birational complexity
of a log Calabi--Yau pair
and the dimension of the associated dual complex.
The following theorem is the main technical statement that we will use to prove
Theorem~\ref{theorem:bcomp-vs-dim}.

\begin{theorem}\label{thm:dim-dc-many-assump}
Let $(X,B,\mathbf{M})$ be a $\qq$-factorial log Calabi--Yau pair. 
Let $X\rightarrow Z$ be a fibration satisfying the following conditions:
\begin{enumerate}
    \item[(i)]  $X\rightarrow Z$ is of Fano type,
    \item [(ii)] every glcc of $(X,B,\mathbf{M})$ dominates $Z$, and 
    \item[(iii)] the relative Picard rank $\rho(X/Z)$ is minimal among $\qq$-factorial varieties $X'$ that admit a birational contraction $X\dashrightarrow X'$ over $Z$.
\end{enumerate}
Let $\Gamma\leqslant \lfloor B\rfloor$ be a reduced divisor. 
Let $r:=|\Gamma|-\rho(X/Z)$ and assume $r\geq 0$. 
Then, there exists a small birational map $\pi\colon X\dashrightarrow X'$ over $Z$ satisfying the following:
\begin{itemize}
    \item $X'$ is $\qq$-factorial,
    \item $\pi_*\Gamma$ admits a decomposition into $i\geq r+1$ prime divisors $B_i$ movable over $Z$, and 
    \item $B_1\cap \dots \cap B_{r} \neq \emptyset$.
\end{itemize}
In particular, we have that $\dim \mathcal{D}(X,B,\mathbf{M})\geq r-1$.
\end{theorem}

\begin{proof}
First, we show the statement when $\rho(X/Z)=1$.
In this case, we may take $X'=X$.
Every component of $\Gamma$ is ample over $Z$ in this case.
Indeed, by (ii) every component of $\Gamma$ dominates $Z$.
Hence, we may consider the decomposition of $\Gamma$ into its $r+1$ irreducible components
in which case it is clear that the intersection of $r$  of them is non-trivial. In particular, we have that $\dim \mathcal{D}(X,B,\mathbf{M})\geq r-1$ by \cite[Proposition 34 and Corollary 38]{dFKX17}.

Now, we assume that $\rho(X/Z)\geq 2$.
We run a $K_X$-MMP over $Z$. 
This MMP must terminate with a Mori fiber space over $Z$. 
Let $X\dashrightarrow X_1 \dashrightarrow \dots \dashrightarrow X_k$ be the steps of the MMP over $Z$
and $X_k\rightarrow W$ be the Mori fiber space over $Z$.
By the assumption (iii), we know that this MMP does not contract divisors, i.e., 
the birational map $X\dashrightarrow X_k$ is small
and $X_k$ is $\qq$-factorial.
Let $(X_k,B_k,\mathbf{M})$ be the generalized log Calabi--Yau pair induced on $X_k$. Let $(W,B_W,\mathbf{N})$ be the generalized log Calabi--Yau pair induced on $W$ by the generalized canonical bundle formula.
Note that $|\Gamma_k|-\rho(X_k/Z)=r$, where $\Gamma_k$ is the image on $X_k$ of $\Gamma$.
Write $\Gamma_{\rm ver}$ for the 
vertical components of $\Gamma_k$,
and $\Gamma_{\rm hor}$ for the horizontal components.
Let $\Gamma_W$ be the reduced sum of the images of $\Gamma_{\rm ver}$ in $W$.
Let $\gamma_0:=|\Gamma_{\rm hor}|-1$.
Then, we have that 
\[
r=|\Gamma_k|-\rho(X_k/Z) = 
|\Gamma_{\rm hor}|+|\Gamma_{\rm ver}|-\rho(X_k/Z) = 
\gamma_0 + 
(|\Gamma_{\rm ver}|-\rho(W/Z)) 
=
\gamma_0 + (|\Gamma_W|-\rho(W/Z)).
\]
Set $r_W:=|\Gamma_W|-\rho(W/Z)$.
Note that $r= \gamma_0+r_W$ and $\Gamma_W\leqslant \lfloor B_W\rfloor$.\\

\textit{Case 1:} We assume that $r_W=0$.\\

In this case, there are $r+1$ components of $\Gamma_k$
that are horizontal over $W$.
Each of these components is movable over $Z$.
Indeed, let $E$ be a component of $\Gamma_k$.
If $E$ is not movable over $Z$, then the $E$-MMP over $Z$ will eventually contract $E$.
This would contradict condition (iii).
Note that we can run such an MMP because of the Fano type assumption (i).
We may set $X':=X_k$.
Hence, the divisor $\Gamma_k$
admits a decomposition into at least $r+1$ prime divisors that are movable over $Z$.
Furthermore, since all these prime divisors are ample over $W$ we conclude that the intersection of $r$ of them 
is non-trivial.
In particular, we have that $\dim \mathcal{D}(X,B,\mathbf{M})\geq r-1$.\\

\textit{Case 2:} We assume that $r_W>0$.\\

Note that $W\rightarrow Z$ is a morphism of Fano type.
We claim that $\rho(W/Z)$ is minimal among $\qq$-factorial varieties that admit a birational contraction $W\dashrightarrow W'$ over $Z$.
Otherwise, there exists a birational contraction $W\dashrightarrow W'$ over $Z$ with $W'$ being $\qq$-factorial and $\rho(W'/Z)<\rho(W/Z)$.
By Lemma~\ref{lem:two-ray-game}, there exists a birational contraction $X_k\dashrightarrow Y$ over $Z$ with $Y$ being $\qq$-factorial and $\rho(Y/Z)<\rho(X/Z)$.
This contadicts our assumption (iii).
Hence, we have that $W\rightarrow Z$ is of Fano type
and $\rho(W/Z)$ is minimal. Recall that $|\Gamma_W|-\rho(W/Z)=r_W>0$ and 
$\Gamma_W\leq \lfloor B_W\rfloor$.
Therefore, the generalized pair $(W,B_W,\mathbf{N})$ satisfies all the assumptions of the theorem and $\dim W<\dim X$.
By induction on dimension, there exists a small birational contraction $\pi_W\colon W\dashrightarrow W'$ over $Z$ 
such that the following conditions are satisfied:
\begin{enumerate}
    \item[(a)] every glcc of $(W',B_{W'},\mathbf{N})$ dominates $Z$, 
    \item[(b)] $\Gamma_{W'}:={\pi_W}_*(\Gamma_W)$ admits a decomposition into $i\geq r_W+1$ prime divisors that are movable over $Z$,  
    \item[(c)] at least $r_W$ prime divisors have non-trivial intersection.
\end{enumerate}
We call $\Gamma_{W',j}$ the prime divisors of the decomposition of ${\pi_W}_*(\Gamma_W)$.
Without loss of generality, assume that the first $r_W$ of these prime divisors intersect non-trivially. By Lemma~\ref{lem:two-ray-game}, there exists a commutative diagram
\[
\xymatrix{
X\ar@{-->}[r]\ar[dd] &  X_k\ar@{-->}[r]\ar[d] & X'\ar[d]^-{\psi}\\
 & W\ar[ld] \ar@{-->}[r] & W'\ar[lld] \\ 
 Z & & 
}
\]
satisfying the following conditions:
\begin{enumerate}
\item $X_k\dashrightarrow X'$ is a small birational contraction;
\item $\psi \colon X' \to W'$ is a Mori fiber space,
\item every glcc of $(X',B',\mathbf{M})$ dominates $Z$, and
\item $\psi^{-1} (\lfloor\Gamma_{W'} \rfloor) \subseteq \Gamma'$,
where $\Gamma'$ is the pushforward of $\Gamma$ on $X'$.
\end{enumerate}
Let $\Gamma_1,\dots,\Gamma_f$ be the prime components of $\Gamma_{\rm hor}$. Recall that $f=\gamma_0+1$.
Arguing as in the previous paragraph we see that each $\Gamma_i$ is movable over $Z$.
Then, the sum 
\begin{equation}\label{eq:sum-gamma}
\Gamma_1+\dots+\Gamma_f+\psi^*\Gamma_{W',1}+\dots+\psi^*\Gamma_{W',r_W+1},
\end{equation}
gives a decomposition of $\Gamma'$
into prime divisors that are movable over $Z$.
Note that $f+r_W+1= \gamma_0+1+r_W+1 \geq r+1$.
Thus, the divisor $\Gamma'$ admits a decomposition
into at least $r+1$ prime divisors that are movable over $Z$.
It suffices to argue that all
the divisors in the sum~\eqref{eq:sum-gamma} have non-empty intersection.
By $(c)$, we know that there exists a point $w\in \Gamma_{W',1}\cap \dots \Gamma_{W',r_W}$.
Since each $\Gamma_i$ is ample over $W$, we conclude that 
$\psi^{-1}(w)\cap \Gamma_1 \cap \dots \cap \Gamma_{f-1}$ is non-empty, equivalently \ $f-1+r_{W} = r$ divisors have non-empty intersection. This finishes the proof.
\end{proof}

\begin{theorem}\label{thm:dim-dc}
Let $(X,B,\mathbf{M})$ be a generalized $\mathbb{Q}$-factorial log Calabi--Yau pair. 
Let $X\rightarrow S$ be a fibration.
Assume that every glcc of $(X,B,\mathbf{M})$ dominates $S$.
Let $r:=|\lfloor B\rfloor|-\rho(X/Z)$
and assume that $r\geq 0$.
Then, there exists a crepant birational model
$(X',B',\mathbf{M})$ of $(X,B,\mathbf{M})$ over $S$
and a fibration $X'\rightarrow Z$ over $S$ satisfying the following:
\begin{itemize}
    \item $X'$ is $\qq$-factorial,
    \item every glcc of $(X',B',\mathbf{M})$ dominates $Z$,
    \item $\lfloor B'\rfloor$ admits a decomposition $\sum_{i\in I} B_i$, with $i\geq r+1$, into prime divisors $B_i$ movable over $Z$, 
    \item $B_1\cap \dots \cap B_r\neq \emptyset$.
\end{itemize}
In particular, we have that $\dim \mathcal{D}(X,B,\mathbf{M})\geq r-1$.
\end{theorem}

\begin{proof}
First, we reduce to the case in which the fibration $X\rightarrow S$ is of Fano type.
By~\cite[Theorem 4.2]{FS20},
 there exists a $\qq$-factorial gdlt crepant birational model $(X_1,B_1,\mathbf{M})$ of $(X,B,\mathbf{M})$ 
and a fibration $X_1\rightarrow W$ over $S$ satisfying the following:
\begin{enumerate}
    \item every glcc of $(X_1,B_1,\mathbf{M})$ dominates $W$, 
    \item every divisor extracted by the birational map $X_1\dashrightarrow X$ is a glcc of $(X,B,\mathbf{M})$, and
    \item $X_1$ is of Fano type over $W$.
\end{enumerate}
Let $(X_2,B_2,\mathbf{M})$ be a $\qq$-factorial gdlt modification of $(X,B,\mathbf{M})$ 
that admits a birational contraction to $X_1$.
By Lemma~\ref{lem:bcomp-dlt-mod}, we have that 
$|\lfloor B_2\rfloor|-\rho(X_2/Z)=r$.
In particular, we conclude that 
$|\lfloor B_1\rfloor|-\rho(X_1/Z)\geq r$.
Thus, we may replace the generalized pair $(X,B,\mathbf{M})$ and the fibration $X\rightarrow S$
with the generalized pair $(X_1,B_1,\mathbf{M})$
and the fibration $X_1\rightarrow W$.
By doing so, we may assume that $X\rightarrow S$ is a morphism of Fano type.

From now on, we assume that the fibration $X\rightarrow S$ is of Fano type.
Let $X_0$ be the variety 
that minimizes $\rho(X_0/Z)$ among
all $\qq$-factorial varieties
that admit a birational contraction
$X\dashrightarrow X_0$ over $Z$.
Let $(X_0,B_0,\mathbf{M})$ be the generalized log Calabi--Yau pair induced on $X_0$.
Note that, by construction, we have that 
\[
|\lfloor B_0\rfloor|-\rho(X_0/Z) \geq 
|\lfloor B\rfloor| -\rho(X/Z) = r. 
\]
We may replace $(X,B,\mathbf{M})$
and the fibration $X\rightarrow S$ 
with $(X_0,B_0,\mathbf{M})$ and the fibration $X_0\rightarrow S$.

Thus, in the statement of the theorem we may further assume:
\begin{enumerate}
    \item the morphism $X\rightarrow S$ is of Fano type, and 
    \item the relative Picard rank $\rho(X/S)$ is minimal.
\end{enumerate}
Hence, the statement follows from Theorem~\ref{thm:dim-dc-many-assump}.
\end{proof}

\begin{proof}[Proof of Theorem~\ref{theorem:bcomp-vs-dim}]
Let $(X,B)$ be a log Calabi--Yau pair with $K_X+B\sim 0$.
Let $c=c_{\rm bir}(X,B)$. 
By Lemma~\ref{lem:computing-bcomp}, there exists a $\qq$-factorial crepant birational dlt model $(X_0,B_0)$ on which 
\[
\dim X_0 + \rho(X_0) - |B_0|=c.
\]
Equivalently, we write $|B_0|-\rho(X_0) = n-c$.
By~\cite[Theorem 4.2]{FS20},
 there exists a crepant birational model $(X_1,B_1,\mathbf{M})$ of $(X_0,B_0,\mathbf{M})$ and a fibration $X_1\rightarrow S$ satisfying the following conditions:
\begin{itemize}
    \item $(X_1,B_1,\mathbf{M})$ is gdlt,
    \item every glcc of $(X_1,B_1,\mathbf{M})$ dominates $S$, 
    \item every exceptional divisor of $X_1\dashrightarrow X_0$ is a glc place of $(X_0,B_0,\mathbf{M})$, and 
    \item $X_1$ is of Fano type over $S$. 
\end{itemize}
By the third aforementioned condition, 
we conclude that $|B_1|-\rho(X_1)\geq n-c$. 
In particular, we have that $|B_1|-\rho(X_1/S)\geq n-c$.
We apply Theorem~\ref{thm:dim-dc} to the pair $(X_1,B_1,\mathbf{M})$ and the fibration $X_1\rightarrow S$. Then, there exist a crepant birational model $(X',B',\mathbf{M})$ of $(X,B,\mathbf{M})$ 
over $S$ and a fibration $X'\rightarrow Z$ over $S$
satisfying the following conditions:
\begin{itemize}
    \item $X'$ is $\qq$-factorial,
    \item every glcc of $(X',B',\mathbf{M})$ dominates $Z$,
    \item $\lfloor B'\rfloor= \sum_{i\in I} B_i$, with $i\geq n-c$, where $B_i$ are prime divisors movable over $Z$,  and
    \item $B_1\cap \dots \cap B_{n-c-1}\neq \emptyset$.
\end{itemize}
In particular, we have that $\dim \mathcal{D}(X,B,\mathbf{M}) \geq n-c-1$. This finishes the proof.
\end{proof}

\section{Dual complexes
of non-maximal birational complexity}

In this section, we explore the connection between the birational complexity and dual complexes. Our first theorem states that dual complexes of non-maximal birational complexity is union of two collapsible subcomplexes.

\begin{theorem}\label{thm:dc-non-max-bc-gen}
Let $(X,B,\mathbf{M})$ be a generalized log Calabi--Yau pair of dimension $n$. Assume that $c_{\rm bir}(X,B,\mathbf{M})<n$, then $\mathcal{D}(X,B,\mathbf{M})$ is the union of two collapsible subcomplexes.
\end{theorem}

\begin{proof}
First, we assume that either $\mathbf{M}$ is non-trivial or $\{B\}$ is non-trivial.
By Theorem~\ref{thm:ample-div-bound}, there exists a crepant birational model $(X',B',\mathbf{M})$ of $(X,B,\mathbf{M})$ and a fibration $X'\rightarrow W$ satisfying the following conditions:
\begin{enumerate}
    \item $X'$ is $\qq$-factorial,
    \item $(X',B',\mathbf{M})$ is gdlt, 
    \item every glcc of $(X',B',\mathbf{M})$ dominates $W$, 
    \item $X'\dashrightarrow X$ only extracts glc places of $(X,B,\mathbf{M})$, and 
    \item $\lfloor B'\rfloor$ fully supports an ample divisor over $W$.
\end{enumerate}
By (4), we have that either $\mathbf{M}$ or $\{B'\}$ is non-trivial.
Let $A$ be the divisor that is fully supported on $\lfloor B'\rfloor$
and is ample over $W$.
We argue that $\mathcal{D}(X',\lfloor B'\rfloor)\simeq_{\mathrm{PL}} \mathcal{D}(X,B,\mathbf{M})$
is collapsible.
We run a $(K_{X'}+\lfloor B' \rfloor)$-MMP with scaling of $A$ over $W$.
Let $X'\dashrightarrow X'_1 \dashrightarrow X'_2 \dashrightarrow \dots \dashrightarrow X'_k$ be the steps of this MMP
and $X'_k\rightarrow Z$ be the associated Mori fiber space.
For each $i$, we denote by $A_i$ the push-forward of $A$ on $X'_i$.
For each $i$, we denote by $B'_i$ the push-forward of $B'$ on $X'_i$. 
Then, the pair $(X'_i,\lfloor B'_i\rfloor)$ is dlt for each $i$.
Since this MMP is with scaling of $A$ over $i$, 
for each step $X'_j\dashrightarrow X'_{j+1}$
the extremal ray $R_j$ satisfies 
$(K_{X'_j}+\lfloor B'_j\rfloor)\cdot R_j <0$ and
$A'_j\cdot R_j>0$.
Note that $A'_j$ is an effective divisor
that is fully supported on $\lfloor B'_j\rfloor$.
Thus, we conclude that $R_j$ must intersect
a prime component of $\lfloor B'_j\rfloor$ positively.
Henceforth, all the conditions of Proposition~\ref{prop:dc-collapse} are satisfied and we conclude that for each $j$
there is a collapse 
\[
\mathcal{D}(X'_j,\lfloor B'_j\rfloor)
\rightarrow 
\mathcal{D}(X'_{j+1},\lfloor B'_{j+1}\rfloor).
\]
Thus, it suffices to show that $\mathcal{D}(X'_k,\lfloor B'_k\rfloor)$ is collapsible. This follows by~\cite[Proposition 24]{KX16}, since $(X'_k, \lfloor B'_{k} \rfloor) \to Z$ is a Fano fibration, and $R_k$ intersects a prime component of $\lfloor B'_{k} \rfloor$ positively, i.e., $\lfloor B'_{k} \rfloor$ dominates $Z$.

From now on, we assume that $\mathbf{M}$ is the trivial b-nef divisor and $B$ is reduced.
By Lemma~\ref{lem:computing-bcomp}, there exists a $\mathbb{Q}$-factorial dlt crepant birational model $(X_0,B_0)$ of $(X,B)$ 
for which $|B_0|-\rho(X_0)>0$.
Hence, we may apply Theorem~\ref{thm:proper-supp-amp} and Lemma~\ref{lem:simplicial} to conclude that there exists
a crepant birational model $(X',B')$ of $(X,B)$ and a fibration $X'\rightarrow W$ satisfying the following:
\begin{enumerate}
    \item $X'$ is $\qq$-factorial,
    \item $(X',B')$ is dlt, 
    \item every lcc of $(X',B')$ dominates $W$,
    \item there exists a decomposition $B' = E + (B'-E)$, where $E$ is a prime divisor and $B'-E$ is the support of an effective divisor that is ample over $Z$, and 
    \item $\mathcal{D}(X',B')$ is simplicial.
\end{enumerate}
Let $F$ be the reduced sum of all the divisors in $B'$ that intersect $E$ non-trivially. Since $\mathcal{D}(X', B')$ is simplicial, the dual complex $\mathcal{D}(X',F)$ is the cone over the dual complex $\mathcal{D}(E,(F-E)|_E)$, so it is collapsible.
Further, we have that
\[
\mathcal{D}(X',B') =
\mathcal{D}(X', B'-E) \cup 
\mathcal{D}(X',F).
\]
Hence, it suffices to show that $\mathcal{D}(X',B'-E)$ is a collapsible complex.
We run a $(K_{X'}+B'-E)$-MMP with scaling of $A$ over $W$, and conclude as in the case that either $\mathbf{M}$ is non-trivial or $\{B\}$ is non-trivial.
\end{proof}

\begin{theorem}\label{thm:dc-smooth-gen}
Let $(X,B,\mathbf{M})$ be a generalized log Calabi--Yau pair of dimension $n$
with $c_{\rm bir}(X,B,\mathbf{M})<n$. If $\mathcal{D}(X,B,\mathbf{M})$ is a PL-manifold, then either 
$\mathcal{D}(X,B,\mathbf{M})\simeq S^k$ or
$\mathcal{D}(X,B,\mathbf{M})\simeq_{\rm PL} D^k$. with $k\leq n-1$. 
In the former case, $\mathcal{D}(X,B,\mathbf{M})\simeq_{\mathrm{PL}} S^k$ if further $\dim \mathcal{D}(X,B,\mathbf{M})\neq 4$.
\end{theorem}

\begin{proof}
If $\mathcal{D}(X,B,\mathbf{M})$ is a closed manifold of positive dimension, then $\mathbf{M}$ is the trivial b-nef divisor and $B$ is reduced. In the notation of the proof of Theorem~\ref{thm:dc-non-max-bc-gen}, we can write 
\[
\mathcal{D}(X',B') =
\mathcal{D}(X', B'-E) \cup 
\mathcal{D}(X',F),
\]
where $\mathcal{D}(X',F)$ is the closed star of the vertex corresponding to the divisor $E$ as $\mathcal{D}(X',B')$ is simplicial. Since $\mathcal{D}(X',B')$ is also a PL-manifold, then $\mathcal{D}(X',F)$ is PL-homeomorphic to a disk. The statement then follows from Proposition~\ref{prop:PL-sphere}.

Now, assume that the manifold has boundary. 
Following the proof of Theorem~\ref{thm:dc-non-max-bc-gen} 
there is a $\qq$-factorial gdlt modification $(X',B',\mathbf{M})$ of $(X,B,\mathbf{M})$ and a component $E$ of $B'$ such that
\[
\mathcal{D}(X',B',\mathbf{M})=
\mathcal{D}(X',B'-E) \cup \mathcal{D}(X',F),
\]
where $F$ is the reduced sum of all the components of $B'$ that intersect $E$ non-trivially.
Furthermore, we know that each 
$\mathcal{D}(X',B'-E)$ and $\mathcal{D}(X',F)$ is a collapsible PL manifold, hence they are PL isomorphic to disks.
If $v_E$ is in the interior of $\mathcal{D}(X',B',\mathbf{M})$, then by further blow-ups, we may assume that 
$\mathcal{D}(X',F) \cap \partial \mathcal{D}(X',B',\mathbf{M})=\emptyset$.
We conclude that 
\[
\partial \mathcal{D}(X',B'-E) = 
\partial \mathcal{D}(X',F) \sqcup 
\partial \mathcal{D}(X',B').
\]
This contradicts the fact that $\mathcal{D}(X',B'-E)$ is a disk.
On the other hand, if $v_E$ is in $\partial \mathcal{D}(X',B')$, then 
$\mathcal{D}(X',B')$ collapses to $\mathcal{D}(X',B'-E)$, so it is collapsible. This finishes the proof.
\end{proof}

\begin{proof}[Proof of Theorem~\ref{theorem:max-bir-comp}]
We consider a finite group $G$ acting freely on $S^{n-1}$.
We may choose a polytope structure $P$ on the sphere that is preserved by $G$.
Let $X(P)$ be the projective toric variety associated to $P$.
Let $B_P$ be the toric boundary of $X(P)$ so $(X(P),B_P)$ is a log Calabi--Yau pair.
Let $(X,B)$ be the quotient $(X(P),B_P)/G$.
Then, $(X,B)$ is a log Calabi--Yau pair of dimension $n$ such that 
$\mathcal{D}(X,B)\simeq_{\rm PL} S^{n-1}/G$; see \cite[Theorem F]{MMS2022}.
In particular, $\mathcal{D}(X,B)$ is a PL-manifold that is not the union of two disks.
Theorem~\ref{theorem:dc-nonmax-bc} implies that $c_{\rm bir}(X,B)=n$.
This finishes the proof.
\end{proof}

\section{Effective cone
in maximal birational complexity}

In this section, we prove a result regarding the effective cone
of divisors of log Calabi--Yau pairs of maximal birational complexity.

\begin{proof}[Proof of Theorem~\ref{theorem:eff-cone-max-bir}]
Let $(X,B)$ be a log Calabi--Yau pair of coregularity zero and dimension $n$.
Assume that $c_{\rm bir}(X,B)=n$.
Up to a crepant birational map, we can suppose that $X$ is of Fano type. 
By Theorem~\ref{thm:mod-with-large-bound} there exists a $\qq$-factorial gdlt crepant birational model $(X_1,B_1)$ of $(X,B)$
whose boundary $B_1$ is reduced with exactly $\rho(X_1)$ components, since $c(X_1, B_1)=\dim X_1 + \rho(X_1)-|B|\geq n$ and $|\lfloor B\rfloor| \geq \rho(X_1)$ by Theorem~\ref{thm:mod-with-large-bound}.(3).
Furthermore, the variety $X_1$ is of Fano type by Lemma~\ref{lem:contr-FT}.

We claim that the effective cone of $X_1$ is simplicial and is generated by the components of $B_1$. 
Since $X_1$ is of Fano type,
the divisor $B_1$ supports a big divisor.
Let $\sigma$ be the cone in $N^1(X_1)$
that is generated by the components of $B_1$. 
Let $\rho$ be the Picard rank of $X_1$.
Note that $\sigma$ intersects the interior of $\overline{\rm Eff}(X_1)$.
It suffices to show the two following statements:
\begin{enumerate}
\item[(i)] the cone $\sigma$ is $\rho$-dimensional, and 
\item[(ii)] the boundary $\partial \sigma$ contains no big divisors.
\end{enumerate}

First we show (i), that the cone $\sigma$ is $\rho$-dimensional.
To do so, it suffices to show that for every two effective divisors $E$ and $F$ supported on $\supp(B_1)$
the equivalence $E\sim_\qq F$ implies that $E=F$.
Assume that $E\neq F$ and $E\sim_\qq F$. We may assume that $E$ and $F$ have no common components. 
We run a $(K_{X_1}+B_1-E)$-MMP 
that terminates with a Mori fiber space $X_k\rightarrow Z_1$.
Since this MMP is $E$-positive,
and so $F$-positive, 
there is a component $E_0$ of $E$ 
and $F_0$ of $F$ that are not contracted.
Furthermore, the push-forwards $E_{0,k}$ and $F_{0,k}$ of $E_0$ and $F_0$ respectively, are horizontal over $Z$.
Let $(X_k,B_k)$ be the log Calabi--Yau pair induced on $X_k$.
Let $(Z_1,B_{Z_1},\mathbf{N})$ be the generalized log Calabi--Yau pair induced on $Z_1$.
Note that ${\rm coreg}(Z_1,B_{Z_1},\mathbf{N})=0$.
Then, by Theorem~\ref{thm:mod-with-large-bound} 
there exists a crepant birational map $\phi \colon (Z_1,B_{Z_1},\mathbf{N}) \dashrightarrow (Z_2,B_{Z_2},\mathbf{N})$, only extracting glc places,
such that the following conditions hold:
\begin{itemize}
\item $(Z_2,B_{Z_2},\mathbf{N})$ is $\qq$-factorial and gdlt, and
\item the number of components of $\lfloor B_{Z_2}\rfloor$ is at least $\rho(Z_2)$.
\end{itemize}
 By Lemma~\ref{cor:mod-base-mod-tot-space} there exists a crepant birational map $\phi \colon (Y_2, B_2) \dashrightarrow (X_{k}, B_k)$, only extracting glc places,
making the following square commutative:
\[
\xymatrix{
(X_k,B_k)\ar[d] & (Y_2,B_{Y_2})\ar[d]^-{\phi} \ar@{-->}[l]
\\
Z_1 & Z_2,\ar@{-->}[l]
}
\]
and the following conditions are satisfied:
\begin{enumerate}
\item  $Y_2$ is $\qq$-factorial,
\item $\phi$ is a Mori fiber space, and 
\item $\phi^{-1}(\lfloor B_{Z_2}\rfloor)\subset B_{Y_2}$.
\end{enumerate}
Note that $X_k\dashrightarrow Y_2$ is an isomorphism over the generic point of $Z_1$, so $\lfloor B_{Y_2}\rfloor$ contains at least two prime components that are horizontal over $Z_2$.
The condition (3) above implies that 
$\lfloor B_{Y_2}\rfloor$ contains at least $\rho(Z_2)$ vertical components over $Z_2$.
Condition (2) above implies that 
$\rho(Y_2)=\rho(Z_2)+1$.
We conclude that $\lfloor B_{Y_2}\rfloor$ contains at least $\rho(Y_2)+1$ components.
This contradicts the fact that $c_{\rm bir}(X_k,B_k)=n$.
We conclude that $\sigma$ is $\rho$-dimensional.

Now, we show (ii), i.e., the boundary $\partial \sigma$ contains no big divisors.
Assume that $\partial \sigma$ contains a big divisor $F$.
The cone $\sigma$ is $\rho$-dimensional and is generated by $\rho$ elements.
Hence, there is a prime component $E$ of $B_1$ that is not contained in the support of $F$.
We run a $(K_{X_1}+B_1-E)$-MMP
that terminates with a Mori fiber space $X_k\dashrightarrow Z$.
The push-forward $E_k$ of $E$ on $X_k$ is horizontal over $Z$.
The divisor $F$ is big and does not contain $E$ on its support.
Then, the push-forward $F_k$ of $F$ on $X_k$
has a prime component $G_k$ that is horizontal over $Z$.
The same argument as in the previous case leads to a contradiction.
We conclude that $\partial \sigma$ contains no big divisor. 
Thus, we conclude that $\sigma={\rm Eff}(X_1)$.
\end{proof}

\section{Examples and questions}

In this section, we collect some examples and questions. Our first example is a special case of Theorem~\ref{theorem:max-bir-comp}.

\begin{example}\label{ex:max-bcomp-dual}
{\em 
Consider the toric variety $T_n:=(\mathbb{P}^1)^n$, with $n \geq 3$, and its toric boundary $B_n$.
Then, the pair $(T_n,B_n)$ is a log Calabi--Yau pair of dimension $n$ with $c(T_n,B_n)=0$.
Consider the involution $i_n$ on $T_n$ given by 
\[
i_n([x_1:y_1],\dots,[x_n:y_n]) =
([y_1:x_1],\dots,[y_n:x_n]).
\]
The involution $i_n$ preserves the toric boundary $B_n$, and permutes its irreducible components.
In fact, for every prime component $S$ of $B_n$ the image $i_n(S)$ is the only other prime component of $B_n$ linearly equivalent to $S$.
Hence, $i_n$ acts trivially on 
$\Pic((\pp^1)^n)\simeq \zz^{2n}$.
We consider the log Calabi--Yau pair
$(Y_n,\Delta_n):=(T_n,B_n)/i_n$.
By the aforementioned, we have that 
\[
\rho(Y_n)=2n \text{ and } |\Delta_n|=n.
\]
We conclude that $c(Y_n,\Delta_n)=n$.
Note that $\mathcal{D}(Y_n,\Delta_n)\simeq \mathbb{P}_\rr^{n-1}$.
By Theorem~\ref{theorem:dc-smooth}, we conclude that $c_{\rm bir}(Y_n,\Delta_n)=n$.
Indeed, the dual complex $\mathcal{D}(Y_n,\Delta_n)$ is a PL manifold that is not PL isomorphic to a sphere or a disk.
}
\end{example} 

Our next example shows that there exist $n$-dimensional log Calabi--Yau pair $(X,B)$
of arbitrary large Picard rank
for which $B$ admits a decomposition into $(n+1)$ big divisors.

\begin{example}\label{ex:wpi-blow-up}
{\em 
Let $X_n:=\pp(1,1,n)$ and $B_n$ be the toric boundary.
Let $B_{1,n},B_{2,n}$ be the two components that pass through the singular point of $X_n$
and $B_{0,n}$ be the remaining component of $B_n$. 
Recall that $B_{0,n}^2=n$.
Let $Y_n\rightarrow X_n$ be the blow-up at $n-1$ points contained in $B_{0,n}$ but not contained in $B_{1,n}$ nor $B_{2,n}$.
Let $(Y_n,\Delta_n)$ be the log pull-back
of $(X_n,B_n)$. Then, $\Delta_n$ admits a decomposition into $3$ big movable divisors. These three big divisors are the strict transforms of the $B_{i,n}$'s. Such divisors are big as they are movable with positive self-intersection. On the other hand, we have that $\rho(Y_n)=n$.
}
\end{example}

In the following example, we show that 
the inequality $c_{\rm bir}(X,B)<2\dim X$ is sharp for pairs. Moreover, the equality can be achieved for generalized pairs. 

\begin{example}\label{ex:maximal-bcomp}
{\em 
Consider the pair 
\[
(X_n,B_{n,m}):=
\left(
(\pp^1)^n, \frac{2}{m}B_{n,m}
\right) 
\]
where $\Delta_{n,m}$ is a general smooth hypersurface of multi-degree $(m,\dots,m)$.
If $m\geq 3$, then the log Calabi--Yau pair $(X_n,B_{n,m})$ is terminal.
From now on, we assume that $m\geq 3$.
Furthermore, every extremal contraction of $X_n$ is a Mori fiber space.
We conclude that $(X_n,B_{n,m})$ has no
non-isomorphic
log Calabi--Yau crepant birational models.
This implies that 
\[
c_{\rm bir}(X_n,B_{n,m})=
c(X_n,B_{n,m})=
\dim X_n + 
\rho(X_n) -
|B_{n,m}| =
n+n - \frac{2}{m} < 2n.
\]
Hence, the birational complexity of $(X_n,B_{n,m})$ is arbitarily close to $2n$.

Now, consider the generalized pair
$(X_n,\mathbf{M}_n)$ where $\mathbf{M}_n$ is the anti-canonical b-divisor. 
Hence, $(X_n,\mathbf{M}_n)$ is a generalized terminal log Calabi--Yau pair.
All the extremal contractions of $X_n$ are Mori fiber spaces.
Thus, the generalized pair $(X_n,\mathbf{M}_n)$ has no non-isomorphic generalized log Calabi--Yau crepant birational models.
This implies that 
\[
c_{\rm bir}(X_n,\mathbf{M}_n)=
c(X_n,\mathbf{M}_n)= \dim X_n + \rho(X_n)  = 2n.
\]
Thus, the birational complexity of a generalized pair $(X,B, \mathbf{M})$ can attain the value $2\dim X$.
}
\end{example}

The following example shows that for smooth Calabi--Yau varieties the birational complexity may be larger than twice the dimension.

\begin{example}\label{ex:smooth-cy}
{\em 
Let $X_n$ be the product of $n$ elliptic curves. Then, $X_n$ is a smooth Calabi--Yau variety. Any crepant birational model of $X_{n}$ is isomorphic to $X_{n}$.
In particular, we have that  $c_{\rm bir}(X_n)=c(X_n)=\dim X_n +\rho(X_n) > 2n$.
}
\end{example}

The following example shows 
the statement of Theorem~\ref{theorem:proper-supp-amp} does not hold if the linear equivalence $K_X+B\sim 0$ is not satisfied. 

\begin{example}\label{ex:supp-ample}
{\em 
Consider the log Calabi--Yau pair 
$(\pp^2,L+\frac{2}{3}E)$ where 
$L$ is a line and $E$ is a general elliptic curve.
Then, the pair $(\pp^2,L+\frac{2}{3}E)$
is plt and its dual complex is a point.
In particular, if $(X,B)$ is any log Calabi--Yau pair which is crepant birational to 
$(\pp^2,L+\frac{2}{3}E)$, then $\lfloor B\rfloor$ has either one or zero components.
In particular, no divisor is properly supported on $\lfloor B \rfloor$.
In particular, $\lfloor B\rfloor$ cannot properly support a relatively ample divisor
for any morphism on any crepant birational model of the pair $(\pp^2,L+\frac{2}{3}E)$.
}
\end{example}

The following example shows that for every $c\in \{0,\dots,n-1\}$, there is a log Calabi--Yau pair of dimension $n$, coregularity $c$, and birational complexity at least $n$.

\begin{example}\label{ex:arbirary-bcomp}
{\em 
Let $n$ be a positive integer.
Let $c\in \{0,\dots,n-1\}$,
Consider the log Calabi--Yau pair 
\[
(X_n,B_n):=(Y_{n-c},\Delta_{n-c})\times \left(\pp^1,\frac{2}{3}\{0\}+\frac{2}{3}\{1\}+\frac{2}{3}\{\infty\} \right)^c, 
\]
where the pair $(Y_{n-c},\Delta_{n-c})$
is as in Example~\ref{ex:max-bcomp-dual}.
By construction, we have that 
${\rm coreg}(X_n,B_n)=c$
and
$\dim(X_n)=n$.
Furthermore, we have that 
$\mathcal{D}(X_n,B_n)\simeq_{\rm PL} \mathbb{P}^{n-c}_\rr$.
By Theorem~\ref{theorem:dc-smooth},
we conclude that $c_{\rm bir}(X_n,B_n)\geq n$.
}
\end{example}

In Theorem~\ref{theorem:dc-smooth}, we assume that $\mathcal{D}(X,B)$ is a PL manifold. In the following example we show that the assumption cannot be dropped. 

\begin{example}\label{ex:suspension}
{\em 
Fix integers $n > m \geq 3$. Consider the log Calabi--Yau pair 
\[(X_{n}, B_{n}) \coloneqq (Y_{m}, \Delta_{m}) \times (\mathbb{P}^{n-m}, \Delta_{\mathbb{P}^{n-m}}),\]
where $(Y_{m},\Delta_{m})$
is the log Calabi--Yau pair in Example~\ref{ex:max-bcomp-dual}, and $\Delta_{\mathbb{P}^{n-m}}$ is the toric boundary of $\mathbb{P}^n$.
Then \[c_{\rm bir}(X_n,B_n) \leq c(X_{n}, B_{n})=c(Y_{m}, \Delta_{m})=m < n.\] 
Further, $\mathcal{D}(X_{n}, B_{n})$ is PL-homeomorphic to the join $ \mathbb{P}_\rr^{m-1}*S^{n-m-1}$, which is not a PL sphere, not even a PL manifold. Though $\mathcal{D}(X_{n}, B_{n})$ can be written as union of two collapsible subcomplexes, PL-isomorphic to $\mathbb{P}_\rr^{m-1}*\mathbb{B}^{n-m-1}$.} 
\end{example}

We conclude this section with a couple of questions that may motivate further research.

\begin{question}\label{quest:bcomp-zero}
Assume that $(X,B)$ is a log Calabi--Yau pair with $c_{\rm bir}(X,B)=0$.
What can we say about $X$ and $B$?
\end{question}

In the setting of the previous question one can conclude that $X$ is a rational variety.
However, it is not clear what can be said about $B$.
In general, the equality $c_{\rm bir}(X,B)=0$ does not allow to control the index of $K_X+B$.
For instance, we can consider $X=\pp^n$ and $B$ the sum of several general hyperplanes with coefficients $>\frac{1}{2}$. 
Then, we can blow-up the variety $X$ along the intersection of the components of $B$ to produce log Calabi--Yau pairs of birational complexity zero with large Picard rank and large index. Most of these examples are somewhat artificial.
If $c_{\rm bir}(X,B)=0$, then it is natural to expect that there exists $\Gamma \leq B$ for which $(X,\Gamma_{\rm red})$ is a log Calabi--Yau pair of index $1$ and birational complexity zero. This happens in all the aforementioned examples.

The following question is related to Theorem~\ref{theorem:big-decomp}.

\begin{question}\label{quest:max-big-sphere}
Let $(X,B)$ be an $n$-dimensional log Calabi--Yau pair such that $B$ admits a decomposition into $n+1$ big divisors. Is the pair $(X,B)$ qdlt? 
Is $\mathcal{D}(X,B)$ isomorphic to the boundary of an $n$-dimensional standard simplex?
\end{question} 

In Theorem~\ref{theorem:big-decomp}, we show that whenever $B$ admits a decomposition into $n+1$ big divisors, the pair $(X,B)$ is crepant birational to a weighted projective space. In particular, $\mathcal{D}(X,B)$ is PL-homeomorphic to a PL-sphere.
However, we expect 
that $\mathcal{D}(X,B)$ is the boundary of an $n$-dimensional simplex, and the pair $(X,B)$ has qdlt singularities on the nose.

\bibliographystyle{habbvr}
\bibliography{references}

\end{document}